\documentclass[11pt]{amsart}

\usepackage{amsmath,amssymb}
\usepackage{graphicx}
\usepackage{fullpage}
\usepackage{enumerate}
\usepackage{enumitem}
\usepackage{pgf,tikz}
\usetikzlibrary{patterns}
\usetikzlibrary{arrows}
\usetikzlibrary{positioning}
\usepackage{thmtools,thm-restate}

\def\COMMENT#1{}
\let\COMMENT=\footnote
\newtheorem{theorem}{Theorem}[section]

\newtheorem{lemma}[theorem]{Lemma}

\newtheorem{claim}[theorem]{Claim}
\newtheorem{fact}[theorem]{Fact}
\newtheorem{question}[theorem]{Question}
\newtheorem{corollary}[theorem]{Corollary}
\newtheorem{conjecture}[theorem]{Conjecture}
\newtheorem{proposition}[theorem]{Proposition}
\newtheorem{remark}[theorem]{Remark}

\def\eps{\varepsilon}
\newenvironment{proofclaim}{\removelastskip\penalty55\medskip\noindent{\it Proof of the claim. }}

\newcommand{\qedclaim}{
\hfill\scalebox{0.6}{$\blacksquare$}
}

\newcommand{\N}{\mathbb{N}}

\newcommand{\floor}[1]{\left\lfloor #1 \right\rfloor}

\newcommand{\Csm}{\mathcal{C}_{\textrm{sm}}}
\newcommand{\Ctri}{\mathcal{C}_{\textrm{tri}}}
\newcommand{\Clg}{\mathcal{C}_{\textrm{lg}}}

\theoremstyle{definition}

\DeclareUnicodeCharacter{2212}{-}

\title{Cycle tilings and $H$-factors in directed graphs}

\author{Theodore Molla and Andrew Treglown}

\thanks{TM: Department of Mathematics and Statistics, University of South Florida, Tampa, FL. 
  Research supported by \indent NSF grant DMS-2154313.
Email: \texttt{molla@usf.edu}. 
\\
\indent  AT: School of Mathematics, University of Birmingham, United Kingdom. 
Research supported by EPSRC grant \indent UKRI1117.
Email: \texttt{a.c.treglown@bham.ac.uk}. 
}

\begin{document}

\begin{abstract}
We prove several results concerning cycle tilings and $H$-factors in digraphs.
We provide a minimum semi-degree condition for forcing a digraph to contain a given spanning collection of vertex-disjoint orientations of cycles. Our result is asymptotically best possible for odd cycles and can be viewed as a digraph analogue of the El-Zahar conjecture.
In addition, we asymptotically determine the minimum degree threshold for forcing an $H$-factor in a digraph for a range of digraphs $H$, including the cases when $H$ is a tree or anti-directed cycle. Furthermore, 
 an asymptotically exact Ore-type result for forcing a transitive tournament factor in a digraph is  proven. Several related open problems are also highlighted.
\end{abstract}

\maketitle


\section{Introduction}
\subsection{Minimum degree conditions forcing $H$-factors}
The focus of this paper is on cycle tilings and  $H$-factors in digraphs. Unless stated otherwise, the digraphs we consider do not have loops and we allow for at most one edge in each direction between any pair of vertices.
For a digraph $G$ we write $V(G)$ for its vertex set and $E(G)$ for its edge set, and define
$|G|:=|V(G)|$. Given  $x,y \in V(G)$, we write $xy$ for the edge directed from $x$ to $y$ in $G$.
Given two (di)graphs $H$  and $G$, an \emph{$H$-tiling} in $G$ 
is a collection of vertex-disjoint copies of $H$ in $G$. An \emph{$H$-factor} in $G$ is an $H$-tiling that
 covers all the vertices of $G$.
 We  write $T_r$ to denote the transitive tournament on $r$ vertices.
 
The following celebrated result of Hajnal and Szemer\'edi~\cite{hs} determines the minimum degree threshold for forcing a $K_r$-factor in a graph.
\begin{theorem}\cite{hs}\label{hs}
Let $n \in \mathbb N$ be divisible by $r \in \mathbb N$. If $G$ is an $n$-vertex  graph with $\delta (G) \geq (1-1/r)n$, then $G$ contains a  $K_r$-factor. Moreover, the bound on  $\delta (G)$  is tight.
\end{theorem}
In the setting of digraphs, there is more than one natural version of minimum degree.
The \emph{minimum (total) degree $\delta (G)$} of a digraph $G$ is the minimum number of edges incident to a vertex in $G$. The \emph{minimum semi-degree $\delta^0(G)$} of a digraph $G$ is the minimum of all the in- and outdegrees of the vertices in $G$. 
In~\cite{treglown}, the second author determined  the minimum semi-degree $\delta^0(G)$ threshold 
for forcing an $H$-factor in a digraph $G$ for any fixed \textit{tournament} $H$ (for sufficiently large digraphs $G$). Moreover, the analogous threshold for  the minimum degree $\delta(G)$ version of the problem was established through results from~\cite{cdkm, cdmt,  wang2}.
For example, in~\cite{cdkm} it was proven that given any $n \in \mathbb N$ divisible by $r$, every $n$-vertex digraph $G$ with $\delta(G)\geq 2(1-1/r)n-1$ contains a $T_r$-factor. 
Viewing a graph as a digraph where every edge is a `double edge',  one immediately sees that this result implies Theorem~\ref{hs}.

Building on a number of earlier results,  K\"uhn and Osthus~\cite{KuhnO} determined, up to an additive constant, the minimum degree threshold for forcing an  $H$-factor in a graph $G$, for \textit{any} fixed graph $H$. 
Our first new result  provides a digraph analogue of the K\"uhn--Osthus
theorem  for a certain type of $H$-factor.
 Given digraphs $H$ and $F$, a \emph{homomorphism} from $H$ into $F$ is a mapping $\phi :V(H) \rightarrow V(F)$ such that $\phi(x)\phi(y) \in E(F)$ for every
$xy \in E(G)$. A \textit{directed path}  $P$  is a path $v_1 \dots v_k$ where
$v_iv_{i+1} \in E(P)$ for all $i \in [k-1]$.

\begin{theorem}\label{mainthm}
Let $H$ be a digraph that has a homomorphism into a directed path.
Given any $\eta>0$, there exists $n_0=n_0(\eta,H)\in \mathbb N$ such that the following holds for all $n \geq n_0$ divisible by $|H|$.
If $G$ is an $n$-vertex digraph with 
$$\delta(G)\geq (1+\eta)n,$$
then $G$ contains an $H$-factor.
\end{theorem}
For many choices of $H$, the minimum degree condition in Theorem~\ref{mainthm} is asymptotically best possible. For example, suppose $H$ is a connected digraph. Let $G$ be the $n$-vertex digraph that is the disjoint union of two complete digraphs $G_1$ and $G_2$ chosen such that 
the sizes of $G_1$ and $G_2$ are as equal as possible under the constraint that $|H|$ does not divide  $|G_1|$ and $|G_2|$. Then $\delta (G)\geq n-4$ and $G$ does not contain an $H$-factor.

A digraph $C$ is an \textit{orientation of a cycle}
(or simply a \textit{cycle})
if its underlying graph is a cycle; in particular, there are no double edges in $C$. A \emph{directed cycle} $C_k$ on $k$ vertices is a cycle $v_1\dots v_kv_1$ where
$v_iv_{i+1}, v_k v_1 \in E(C_k)$ for each $i \in [k-1]$.
We say that a cycle $C$ is \textit{balanced} if, when traversing $C$, the number of forward edges equals the number of backward edges; thus, if $C$ is balanced it must have even order. 
Well-studied balanced cycles include 
 anti-directed cycles. It is straightforward to see that
balanced cycles have homomorphisms into directed paths, and so we obtain the following corollary.

\begin{corollary}\label{cor1}
Let $C$ be a balanced cycle.
Given any $\eta>0$, there exists $n_0=n_0(\eta,C)\in \mathbb N$ such that the following holds for all $n \geq n_0$ divisible by $|C|$.
If $G$ is an $n$-vertex digraph with 
$$\delta(G)\geq (1+\eta)n,$$
then $G$ contains a $C$-factor.
\end{corollary}
It would be interesting to determine the minimum degree threshold for forcing a $C$-factor in a digraph for other orientations of a cycle $C$.
For directed cycles this threshold is known. Indeed, 
Wang~\cite{wang2} showed that if $n\in \mathbb N$ is divisible by $3$, then every $n$-vertex digraph $G$ with $\delta (G) \geq (3n-3)/2$ contains a $C_3$-factor.
Zhang and Wang~\cite{zw} then showed the same condition on $\delta (G)$ forces a $C_4$-factor (provided $4$ divides $n$).\footnote{In fact, their result is more refined  and shows there is a unique extremal example.}
Finally, Czygrinow, Kierstead and Molla (see~\cite[Corollary 1.5.7]{theothesis}) proved that 
for any $k \geq 3$, there exists
$n_0 \in \mathbb N$ such that the following holds: if $G$ is a digraph on $n \geq n_0$ vertices, $n$ is divisible
by $k$ and $\delta (G) \geq (3n-3)/2$, then 
$G$ can be partitioned into tiles of order $k$ such that each
tile contains every orientation of a cycle on $k$ vertices; in particular, $G$ contains a $C$-factor for any orientation of a cycle $C$ on $k$ vertices.\footnote{See~\cite{wang3} for a related result.}
Moreover, there are examples that show that
 one cannot lower the bound on $\delta(G)$ here in the case when $C=C_k$; see, e.g.,~\cite[Example 1.3.4]{theothesis}.

 On the other hand, for orientations of a cycle $C$ that are not directed, a result of Lo implies that the minimum degree threshold for forcing a $C$-factor in an $n$-vertex digraph `tends to $(1+o(1))n$' as $|C|$ grows; see~\cite[Corollary 3.3]{lonote} for the precise statement.

\smallskip

There has been much recent interest in minimum semi-degree conditions for forcing a fixed or spanning tree in digraphs; see, e.g.,~\cite{mont, tassio, stein, stein2}. 
Note that Theorem~\ref{mainthm} has implications to tree factors since 
any  tree has a homomorphism into a sufficient long directed path.

\begin{corollary}\label{cor2}
Let $T$ be a tree.
Given any $\eta>0$, there exists $n_0=n_0(\eta,T)\in \mathbb N$ such that the following holds for all $n \geq n_0$ divisible by $|T|$.
If $G$ is an $n$-vertex digraph with 
$$\delta(G)\geq (1+\eta)n,$$
then $G$ contains a $T$-factor.
\end{corollary}

\subsection{Cycle tilings and an analogue of the El-Zahar conjecture}

Since $\delta(G)\geq 2 \delta ^0(G)$ for every digraph $G$, Theorem~\ref{mainthm} yields the asymptotically sharp minimum semi-degree condition for forcing an $H$-factor in a digraph, for any connected digraph $H$ that has a homomorphism into a directed path.\footnote{The case when $H$ is an anti-directed cycle was already established in~\cite[Observation 1.1]{debiasiomolla}.}

The next theorem asymptotically determines the minimum semi-degree threshold for forcing a $C$-factor for any orientation of an \emph{odd} cycle $C$.

\begin{theorem}\label{elzaharspecial}
Given $k\in \mathbb N$ and $\eta >0$, there exists
$n_0=n_0(k,\eta)\in \mathbb N$ such that the following holds for all $n \geq n_0$ divisible by $2k+1$. Let $C$ be an orientation of a cycle on $2k+1$ vertices. If $G$ is an $n$-vertex digraph with
$$\delta^0 (G) \geq \frac{(k+1)n}{2k+1}+\eta n,$$
then $G$ contains a $C$-factor.
\end{theorem}

Given any orientation of a cycle $C$ of length $2k+1$
and $n \in \mathbb N$ divisible by $2k+1$, let $G$ be the complete $3$-partite   digraph with vertex classes of sizes $\frac{n}{2k+1}-1, \frac{kn}{2k+1}+1$ and $\frac{kn}{2k+1}$. Then $G$ does not contain a $C$-factor and $\delta^0(G) =\frac{(k+1)n}{2k+1}-1$. 
Thus, the minimum semi-degree condition in Theorem~\ref{elzaharspecial} is asymptotically sharp.

Note that the underlying graph of $G$ is an extremal example for the K\"uhn--Osthus theorem~\cite{KuhnO} in the case of $C_{2k+1}$-factors -- so Theorem~\ref{elzaharspecial} aligns with the corresponding behaviour exhibited in the graph setting.
Exact versions of Theorem~\ref{elzaharspecial} are known in the case when $C=T_3$ or $C_3$; indeed, in these cases it is known that
if $G$ is a sufficiently large $n$-vertex digraph with
$\delta^0(G)\geq 2n/3$, then $G$ contains a $C$-factor; see~\cite{cdkm, ckm, treglown}.\footnote{In fact, the condition that $n$ is sufficiently large is not needed for $C=T_3$.}

Rather than proving Theorem~\ref{elzaharspecial}  directly, we will prove a significant generalisation that relates to a classical problem in the graph setting.
The well-known El-Zahar conjecture~\cite{zahar} provides a minimum degree condition for covering a graph with vertex-disjoint cycles.
\begin{conjecture}\cite{zahar} \label{conjelzahar}
    Suppose that $G$ is an $n$-vertex graph and $n_1,\dots, n_r\geq 3$ are integers such that
    $\sum_{i=1}^{r}n_i=n$. If 
    $$\delta(G) \geq \sum_{i=1}^{r} \lceil n_i/2 \rceil, $$
    then $G$ contains $r$ vertex-disjoint cycles whose lengths are $n_1,\dots, n_r$.
\end{conjecture}
In 1998 the El-Zahar conjecture was proven by Abbasi~\cite{abbasi}  for sufficiently large graphs. Note that
in the case when the number of odd $n_i$s is $t \geq 2$, the minimum degree condition in Conjecture~\ref{conjelzahar} can be seen to be best possible by considering the complete $3$-partite graph with vertex classes of sizes 
$t-1$, $(n-t+2)/2$ and $(n-t)/2$.

Our next result gives an El-Zahar-type result in the directed setting.
\begin{theorem}\label{thm:elzthm}
  For any $\eta > 0$, there exists $n_{0} = n_{0}(\eta)\in \mathbb N$ such that the following holds.
  Let $\ell_1,  \dotsc, \ell_t\geq 3$ be integers
  such that $n := \ell_1 + \dotsm + \ell_t \ge n_{0}$, and let $D_1,\dots, D_t$ be  orientations of cycles of lengths $\ell_1,\dots, \ell _t$, respectively.
  If $G$ is an $n$-vertex digraph with
  \[
    \delta^0(G) \ge \frac{n + t}{2} + \eta n,
  \]
  then 
  $G$ contains vertex-disjoint copies of $D_1,\dots, D_t$.
\end{theorem}
Notice that Theorem~\ref{elzaharspecial} is a special case of Theorem~\ref{thm:elzthm}.
If each cycle $D_i$ has odd length, then
the minimum semi-degree condition in Theorem~\ref{thm:elzthm} agrees with
the degree condition in Conjecture~\ref{conjelzahar} up to the $\eta n$ term. Therefore by considering the digraph analogue of the extremal example for Conjecture~\ref{conjelzahar}, one sees that Theorem~\ref{thm:elzthm} is asymptotically best possible for odd cycles. 
It would be extremely interesting to establish a full digraph analogue of the El-Zahar conjecture that is sharp for all choices of the cycles $D_i$.

\subsection{Ore-type results for transitive tournament factors}
Given a graph $G$ and $x \in V(G)$, we write $d_G(x)$ for the \textit{degree} of $x$ in $G$.
Given a digraph $G$ and  $x \in V(G)$, we write $d^+_G(x)$ and $d^-_G(x)$ for the \textit{out-degree} and \textit{in-degree}  of $x$ in $G$, respectively. 

A result of Kierstead and Kostochka~\cite{kier} on equitable colourings 
yields the following Ore-type version of Theorem~\ref{hs}.
\begin{theorem}\cite{kier} \label{orehsz}
    Let $n \in \mathbb N$ be divisible by $r \in \mathbb N$.
    If $G$ is an $n$-vertex graph so that 
    $$d_G(x)+d_G(y) \geq 2(1-1/r)n-1$$ 
for all non-adjacent $x\not =y \in V(G)$, then $G$ contains a $K_r$-factor.
\end{theorem}
 Note that Theorem~\ref{orehsz} immediately implies Theorem~\ref{hs} and the Ore-type condition cannot be lowered. The $r=3$ case of Theorem~\ref{orehsz} was established earlier by Enomoto~\cite{enomoto}
 and Wang~\cite{wang}.

 Ore-type results have also been studied for digraphs. For example, a well-known result of Woodall~\cite{wood} from 1972 states the following: if $G$ is a digraph on $n \geq 3$ vertices so that
$d^+_G(x) + d^-_G (y) \geq n$ 
for every $x \not = y \in V(G)$ with $xy \not \in E(G)$, then $G$ contains a directed Hamilton cycle.

The next result similarly provides an Ore-type condition for forcing a $T_3$-factor in a digraph. 

\begin{theorem} \label{orethm}
Let $n \in \mathbb N$ be divisible by $3$.
 Suppose that $G$ is an $n$-vertex digraph so that for every $x \not = y\in V(G)$ with $xy \notin E(G)$ we have 
\begin{align}\label{conmin}
 \ d ^+_G (x)+ d^-_G(y) \geq  4n/3-1.
\end{align}
Then $G$ contains a $T_3$-factor.
\end{theorem}

Note that the Ore-type condition in Theorem~\ref{orethm} cannot be lowered. For example, consider the digraph $G$ consisting of vertex classes $V_1, V_2$ together with a vertex $w$ such that:
$|V_1|=2n/3-1$ and $|V_2|=n/3$; there are all possible (double) edges (i) in $V_1$, (ii) between $V_1$ and $V_2$ and (iii) between $w$ and $V_2$. 
Then for any  $x \not = y\in V(G)$ such that $xy \notin E(G)$ we have 
$ d ^+_G (x)+ d^-_G(y) \geq  4n/3-2$. However, $G$ does not contain a $T_3$-factor since $w$ does not lie in a copy of $T_3$.

Our next result asymptotically generalises Theorem~\ref{orethm} to $T_r$-factors for all $r\geq 2$. 
\begin{theorem} \label{orethm2}
Let $r \geq 2$ be an integer and let $\eta >0$. There exists $n_0=n_0(r,\eta) \in \mathbb N$ such that the following holds for all $n \geq n_0$ divisible by $r$.
 Suppose that $G$ is an $n$-vertex digraph  so that for any $x \ne y\in V(G)$ such that $xy \notin E(G)$ we have that 
\begin{align}\label{oreuseful}
 \ d ^+_G (x)+ d^-_G(y) \geq  2(1-1/r+\eta)n.
\end{align}
Then $G$ contains a $T_r$-factor.
\end{theorem}

Note that one cannot replace the $\eta n$ term in Theorem~\ref{orethm2} with a $-2$. 
 Indeed, consider the 
complete $r$-partite digraph $G$ with $r-2$ vertex classes of size $n/r$, one vertex class of size $n/r-1$ and one vertex class of size $n/r+1$. Then $G$ does not contain a $T_r$-factor and 
$d ^+_G (x)+ d^-_G(y) \geq   2(1-1/r)n-2$ for all   $x ,y\in V(G)$.

\subsection*{Notation}
Given a digraph $G$ and  $x \in V(G)$, we write $N^+_G(x)$ and $N^-_G(x)$ for the \textit{out-neighbourhood} and \textit{in-neighbourhood}  of $x$ in $G$, respectively.
For disjoint sets $A,B \subseteq V(G)$, we write $e_G(A,B)$ for the number of edges in $G$ with startpoint in $A$ and endpoint in $B$.
Given a set $X\subseteq V(G)$, we write $G[X]$ for the subdigraph of $G$ induced by $X$ and let $G\setminus X:=G[V(G)\setminus X]$. Given 
$x_1,\dots, x_t \in V(G)$, we
set $G[x_1,\dots, x_t]:=G[\{x_1,\dots, x_t\}]$. 
Given a set $X\subseteq V(G)$ and $v  \in V(G)$, we define $d^+_G(v,X):=|N^+_G(v) \cap X|$ and $d^-_G(v,X):=|N^-_G(v) \cap X|$.

If $G$ is a digraph with loops and $z \in V(G)$ we take the convention that $d^+_G(z)$ (resp. $d^-_G(z)$) is the number of out-neighbours (resp. in-neighbours)  of $z$ in $G$ excluding $z$ itself.

Let $\mathcal H$ be a collection of digraphs and $G$ be a digraph. We say that $G$ contains an \emph{$\mathcal H$-factor} if $G$ contains a collection of vertex-disjoint copies of elements
from $\mathcal H$ that together cover all the vertices of $G$. 

The constants in the hierarchies used to state our results are chosen from right to left.
For example, if we claim that a result holds whenever $0< a\ll b\ll c\le 1$, then 
there are non-decreasing functions $f:(0,1]\to (0,1]$ and $g:(0,1]\to (0,1]$ such that the result holds
for all $0<a,b,c\le 1$  with $b\le f(c)$ and $a\le g(b)$.  
Note that $a \ll b$ implies that we may assume in the proof that, e.g., $a < b$ or $a < b^2$.

\subsection*{Organisation of the paper}
The paper is organised as follows. In Section~\ref{sec:pre} we introduce some useful results, including the diregularity lemma~\cite{AlonShapira04} and an absorbing result of Lo and Markstr\"om~\cite{lomark}.
We then prove Theorem~\ref{mainthm} in Section~\ref{sec:main}.
We prove Theorem~\ref{thm:elzthm} in Section~\ref{sec:elzahar}.
In Section~\ref{sec:triangle} we prove Theorem~\ref{orethm} before proving Theorem~\ref{orethm2} in Section~\ref{sec:ore2}.
Finally, we raise some open problems in Section~\ref{seccon}.


\section{Useful results}\label{sec:pre}
\subsection{The diregularity lemma}

Let $G$ be a digraph and $A,B\subseteq V(G)$ be disjoint. The \emph{density of~$(A,B)$} is defined by $d_G(A,B):=\tfrac{e_G(A,B)}{|A||B|}$; so $d_G(A,B)$ is not necessarily equal to~$d_G(B,A)$.
Given~$\eps>0$  we say that~$(A,B)$ is~\emph{$\eps$-regular (in~$G$)} if for all subsets~$A'\subseteq A$ and~$B' \subseteq B$ with~$|A'|> \eps |A|$ and~$|B'|> \eps |B|$ we have $$|d_G(A,B)-d_G(A', B')|<\eps\,.$$

We will make use of the following well-known property.
\begin{proposition}\label{prop:regsubset}
Suppose that~$0<\eps < \xi \leq 1/2$.
Let $(A,B)$ be~$\eps$-regular with density~$d$. 
If~$A'\subseteq A$ and~$B'\subseteq B$ with~$|A'|\geq \xi |A|$ and $|B'|\geq \xi |B|$ then~$(A',B')$ is~$\eps/\xi$-regular with density at least~$d-\eps$. \qed 
\end{proposition}


We now state the  \emph{degree form} of the diregularity lemma. 

\begin{lemma}[Diregularity lemma~\cite{AlonShapira04}]\label{lem:reg}

Given any~$\eps\in (0,1)$ and~$\ell_0\in \mathbb N$, there exist $L=L(\eps,\ell_0) \in \mathbb N$ and $n_0=n_0(\eps, \ell_0)\in \mathbb N$ such that the following holds for all $n \geq n_0$. 
Let~$G$ be an $n$-vertex digraph   and let~$d\in[0,1]$. 
Then, there is a partition~$\{V_0, V_1,\dots, V_\ell\}$ of~$V(G)$ with~$\ell_0<\ell<L$ and a spanning subdigraph~$G'$ of $ G$ such that
\begin{enumerate}[label={\rm (\alph*)}]
    \item \label{it:trash}$|V_0|\leq \eps n$; 
    \item $|V_i|=|V_1|$ for every~$i\in [\ell]$; 
    \item for every $v \in V(G)$, $d^{+}_{G'}(v)>d^{+}_G(v)-(d+\eps)n$ and~$d^{-}_{G'}(v)>d^{-}_G(v)-(d+\eps)n$;
    \item $e(G'[V_i])=0$ for every~$i\in [\ell]$; 
    \item \label{it:reg} for every distinct~$i,j\in [\ell]$, the pair~$(V_i,V_j)$ is~$\eps$-regular in~$G'$ with density either~$0$ or at least~$d$. 
\end{enumerate}
\end{lemma}
We call $V_1,\dots, V_\ell$  \emph{clusters} and $V_0$ the \textit{exceptional set},  and refer to $G'$ as the \emph{pure digraph}. The \emph{reduced digraph $R$ of $G$} with parameters $\eps$, $d$ and $\ell_0$ is the digraph
defined by
\begin{align*}
  V(R)&:=\{V_1,\dots, V_t\} \qquad \text{and} \qquad E(R):=\{V_iV_j \colon d_{G'}(V_i,V_j) \ge d   \}. 
\end{align*}
The \emph{special reduced digraph $R^*$ of $G$} with parameters $\eps$, $d$ and $\ell_0$ is the 
digraph obtained from  the reduced digraph $R$ of $G$ by adding a loop at every $V_i \in V(R)$ 
such that $G[V_i]$ is a complete digraph.
The notion of the special reduced digraph will be crucial for our proof of Theorem~\ref{orethm2}.

The following well-known result  states that the reduced digraph of $G$ essentially `inherits' any lower bound on the minimum degree of $G$. 

\begin{proposition}\label{prop:reduceddegree}
Let~$0<\eps\leq d/2$ and
    let $G$ be an $n$-vertex digraph such that $\delta (G)\geq \alpha n$ for some $\alpha >0$.
    Suppose we have applied Lemma~\ref{lem:reg} to $G$ to obtain the reduced digraph $R$ of $G$ with parameters $\eps$, $d$ and $\ell_0$. Then 
    $\delta(R)\geq (\alpha-4d)|R|\,.$\qed
\end{proposition}

Given a digraph $R$ and $t\in\mathbb{N}$, we let $R(t)$ denote the~\emph{$t$-blow-up of~$R$}. 
More precisely, $V(R(t)):= \{v^{j} \colon v \in V(R) \text{ and } j\in [t]\}$ and $E(R(t)) := \{v^a w^{b} :  v w \in E(R) \text{ and }a,b\in [t]\}.$

The following result is a special case
of the counting lemma (often called the key lemma) from~\cite{simo}.
\begin{lemma}\cite{simo}\label{simo}
Suppose that $0<\eps < d$, that $m,t \in \mathbb N$ and that $R=v_1\dots v_k$ is a directed path. Construct a digraph $G$ by replacing every vertex $v_i \in V(R)$ by a set $V_i$ of $m$ vertices, and replacing
each edge $v_i v_{i+1}$ of $R$ with an $\eps$-regular pair $(V_i,V_{i+1})$ of density at least $d$. For each $v_i \in V(R)$ let $U_i$ denote the set of $t$ vertices in $R(t)$ corresponding to $v_i$. 
Let $H$ be a subdigraph of $R(t)$ on $h$ vertices and maximum degree $\Delta \in \mathbb N$. Set $\delta :=d-\eps$ and $\eps _0 := \delta ^\Delta /(2+\Delta)$. If $\eps \leq \eps _0$ and $t-1 \leq \eps _0 m$ then
there are at least $(\eps _0 m)^h$ copies of $H$ in $G$
so that if $x \in V(H)$ lies in $U_i$, then $x$ is embedded into $V_i$ in $G$.
\end{lemma}

The following  result will be applied in the proof of Theorem~\ref{orethm2} to convert a large $T_r$-tiling in a special reduced digraph into a large $T _r$-tiling in the original digraph $G$. It is (for example) a special case of Corollary 2.3 in~\cite{alony}.
\begin{lemma}\label{newlemma}
Let $\eps, d>0$ and $m,r \in \mathbb N$ such that $0<1/m \ll \eps \ll d \ll 1/r$. Let $H$ be a digraph obtained from $T_r$ by replacing every vertex of $T_r$ with $m$ vertices and replacing each edge of $T_r$ with an $\eps ^2$-regular pair of density at least $d$. Then $H$
contains a $T _r$-tiling covering all but at most $\eps m r$ vertices.
\end{lemma}

\subsection{The absorbing lemma}
Let $H$ be a digraph. Given a digraph $G$, a set $S \subseteq V(G)$ is called an \emph{$H$-absorbing set for $Q \subseteq V(G)$}, if both
$G[S]$ and $G[S\cup Q]$ contain  $H$-factors. 

The proof of Theorem~\ref{mainthm} makes use of the following \textit{absorbing lemma} of Lo and Markstr\"om~\cite{lomark}.
\begin{lemma}\cite{lomark}\label{lolemma}
Let $h,t \in \mathbb N$ and let $\gamma >0$. Suppose that $H$ is an $h$-vertex digraph. Then there exists an $n_0 \in \mathbb N$ such that the following holds. Suppose that $G$ is a digraph
on $n \geq n_0$ vertices so that, for any $x,y \in V(G)$, there are at least $\gamma n^{th-1}$ $(th-1)$-sets $X \subseteq V(G)$ such that both $G[X \cup \{x\}]$ and $G[X \cup \{y\}]$ contain  $H$-factors.
Then $V(G)$ contains a set $M$ so that
\begin{itemize}
\item $|M|\leq (\gamma/2)^h n/4$;
\item $M$ is an $H$-absorbing set for any $W \subseteq V(G) \setminus M$ such that $|W| \in h \mathbb N$ and  $|W|\leq (\gamma /2)^{2h} hn/32 $.
\end{itemize}
\end{lemma} 

\subsection{Probabilistic estimates}
We will need the following form of the well-known Chernoff-Hoeffding inequality.
\begin{lemma}\label{chernoff}
Let $X$ be a nonnegative random variable with binomial or hypergeometric distribution and expected value $\mu$. 
Then, for $t \ge 0$
\begin{equation}\label{eq:chernoff}
\mathbb{P}[X \le \mu - t] \le e^{-t^2/2\mu}.
\end{equation}
\end{lemma}

\section{Proof of Theorem~\ref{mainthm}}\label{sec:main}
Let $P(t_1,\dots, t_k)$ denote the blow-up of the directed path $v_1\dots v_k$ where 
$v_i$ is replaced by a set of $t_i$ vertices for each $i \in [k]$. If $H$ is a digraph with a homomorphism into a directed path, then $H$ is a spanning subdigraph of $P(t_1,\dots, t_k)$
for some choice of $k, t_1,\dots, t_k \in \mathbb N$.
Thus, to prove Theorem~\ref{mainthm} it suffices to prove it in the case when $H=P(t_1,\dots, t_k)$.
We first prove an approximate version of Theorem~\ref{mainthm} in this case.

\begin{theorem}\label{approxmainthm}
Let $k \geq 2$ and $t_1,\dots, t_k \in \mathbb N$. Set $H:=P(t_1,\dots, t_k)$.
Given any $\eta>0$, there exists $n_0=n_0(\eta,H)\in \mathbb N$ such that the following holds for all $n \geq n_0$.
If $G$ is an $n$-vertex digraph with 
$$\delta(G)\geq (1+\eta)n,$$
then $G$ contains an $H$-tiling covering all but at most $\eta n$ vertices.
\end{theorem}
\begin{proof}
Define  constants $\eps , d >0$ and  $n_0, \ell _0 , L \in \mathbb N$ so that
\begin{align}\label{hier}
    \frac{1}{n_0} \ll \frac{1}{L} \ll \frac{1}{\ell_0} \ll \eps \ll d \ll \eta, \frac{1}{|H|}.
\end{align}
Let $G$ be a digraph on $n \geq n_0$ vertices as in the statement of the theorem. Apply Lemma~\ref{lem:reg} with
parameters $\eps, d, \ell_0$ to obtain the reduced digraph $R$ of $G$ with $\ell_0 < |R| < L$, and the pure digraph $G'$.
Let $m$ denote the size of the clusters of $G$ and set $\ell:=|R|$.

Proposition~\ref{prop:reduceddegree} and (\ref{hier}) imply that $\delta (R) \geq (1+\eta/2)|R|$.
Therefore, Ghouila-Houri's theorem~\cite{gh} implies that $R$ contains a directed Hamilton path.

Initially, suppose $R$ in fact contains a directed Hamilton cycle $V_1\dots V_\ell V_1$; so the clusters of $G$ are $V_1,\dots, V_\ell$. Set $V_{\ell+1}:=V_1$.

 Lemma~\ref{simo} implies that $G[V_1\cup \dots \cup V_k]$ contains a copy of $H$ with precisely
$t_i$ vertices in $V_i$ for each $i \in [k]$.
Similarly, Lemma~\ref{simo} implies that $G[V_2\cup \dots \cup V_{k+1}]$ contains a copy of $H$ with precisely
$t_i$ vertices in $V_{i+1}$ for each $i \in [k]$.
Repeating this process (shifting the clusters we find the copy of $H$ in each time), one can find a collection $\mathcal C_1$ of $\ell$ vertex-disjoint copies of $H$ in $G$
that cover precisely the same number  of vertices (i.e., $|H|$) in each of the clusters $V_i$ (for $i \in [\ell]$).

Remove the vertices of $\mathcal C_1$ from the clusters $V_1,\dots , V_{\ell}$. Then Proposition~\ref{prop:regsubset} implies that $(V_i,V_{i+1})$ is still an $\eps/2$-regular pair
 in $G'$ with density at least $d-\eps$.
We may now again repeatedly apply Lemma~\ref{simo} to find a collection $\mathcal C_2$ of $\ell$ vertex-disjoint copies of $H$ in $G$
that cover precisely the same number of vertices in each of the clusters $V_i$ (for $i \in [\ell]$).

In fact, notice that we can repeat this process until we have covered all but at most
$\eta m/3$ vertices from each of the clusters $V_i$ of $G$ with vertex-disjoint copies of $H$. Indeed, 
given any $i \in [\ell]$ and $V'_i\subseteq V_i$ and $V'_{i+1}\subseteq V_{i+1}$
such that $|V'_i|, |V'_{i+1}| \geq \eta m/3$, Proposition~\ref{prop:regsubset} implies that
 $(V'_i,V'_{i+1})$ is a $(3\eps/\eta)$-regular pair in $G'$ 
with density at least $d-\eps$. Since $3\eps/\eta \leq \sqrt{\eps}$, this allows one to still
apply Lemma~\ref{simo} where needed (where now $\sqrt{\eps}$ plays the role of $\eps$ and $d-\eps$ plays the role of $d$).

In summary, this process ensures we can obtain an $H$-tiling $\mathcal H$ in $G$ that does not cover at most
$\eta m/3$ vertices from each cluster $V_i$.  As $\mathcal H$ also does not cover any vertex from $V_0$,
 $\mathcal H$ is an $H$-tiling in $G$ covering all but at most 
$$|V_0|+ \ell \cdot \frac{\eta m}{3} \leq \eps n +\frac{\eta n}{3} 
\stackrel{(\ref{hier})}{\leq} 
\frac{\eta n}{2}$$
vertices of $G$, as desired.

Now suppose $V_1\dots V_\ell$ is just a directed Hamilton path in $R$ rather than a directed Hamilton cycle. Then we can follow precisely the same procedure as above, except that at each step where we require
a copy $H'$ of $H$ in $G$ that uses both vertices from $V_1$ and $V_{\ell}$, we cannot find such an $H'$.
In the above procedure we only covered at most 
$$ |V_1|+\dots +|V_{k-1}| +|V_{\ell-k+2}|+\dots +|V_{\ell}| = (2k-2)m \leq (2k-2) n/\ell  \stackrel{(\ref{hier})}{\leq}  \eps n$$
vertices using such copies $H'$ of $H$. By ignoring such copies of $H$ in the $H$-tiling $\mathcal H$, this tells us that we can still find an $H$-tiling in $G$ covering all but at most
$
\eta n/{2} +\eps n \leq \eta n
$
vertices of $G$, as desired.
\end{proof}

\begin{remark}
The reader may wonder why we did not  prove  Theorem~\ref{approxmainthm} via a {single} application of the blow-up lemma~\cite{blowup}. Note though, in general one cannot apply the blow-up lemma to \textbf{spanning} structures of the reduced digraph (such as directed Hamilton cycles or paths). More formally, in an application of the blow-up lemma, one  requires that the number of clusters $T$ considered satisfies $\eps \ll 1/T$, but in our case $T=\ell$ and $1/\ell \ll \eps$.
\end{remark}

Next we apply Lemma~\ref{lolemma} to provide an absorbing lemma for Theorem~\ref{mainthm}.
\begin{lemma}\label{absorbmainthm}
Let $k \geq 2$ and $t_1,\dots, t_k \in \mathbb N$. Set $H:=P(t_1,\dots, t_k)$ and $h:=|H|$.
Given any $\eta>0$, there exist $\xi =\xi(\eta,H)>0$ and $n_0=n_0(\eta,H)\in \mathbb N$ such that the following holds for all $n \geq n_0$.
If $G$ is an $n$-vertex digraph with 
$$\delta(G)\geq (1+\eta)n,$$
then $V(G)$ contains a set $M$ so that 
\begin{itemize}
\item $|M|\leq \xi n$;
\item $M$ is an $H$-absorbing set for any $W \subseteq V(G) \setminus M$ such that $|W| \in h\mathbb N$ and  $|W|\leq \xi ^2 hn/2 $.
\end{itemize}
\end{lemma}
\begin{proof}
Define  constants $\eps , d >0$ and  $n_0, \ell _0 , L \in \mathbb N$ so that
\begin{align}\label{hier3}
    \frac{1}{n_0} \ll \frac{1}{L} \ll \frac{1}{\ell_0} \ll \eps \ll d \ll \eta, \frac{1}{h}.
\end{align}
Let $G$ be a digraph on $n \geq n_0$ vertices as in the statement of the lemma.

Consider any $x,y \in V(G)$. As $\delta (G) \geq (1+\eta) n$, $|N^+_G(x) \cap N^+_G(y)| \geq \eta n$ or $|N^-_G(x) \cap N^-_G(y)| \geq \eta n$. We may assume that the former holds since the latter case follows analogously. 

Apply Lemma~\ref{lem:reg} with
parameters $\eps, d, \ell_0$ to obtain the reduced digraph $R$ of $G$ with $\ell_0 < |R| < L$, and the pure digraph $G'$.
Let $m$ denote the size of the clusters of $G$. 

Note that there exists some cluster $V_2$  such that $V'_2 :=V_2 \cap N^+_G(x) \cap N^+_G(y)$ satisfies 
$|V'_2|\geq \eta m/2$. Moreover, Proposition~\ref{prop:reduceddegree} and (\ref{hier3}) imply that $\delta (R) \geq (1+\eta/2)|R|$. Thus, there is a directed path $V_1V_2\dots V_k$ in $R$ on $k$ vertices whose second vertex is $V_2$.

For each $i \in [k]\setminus \{2\}$, define $V'_i$ to be a  subset of $V_i$ of size precisely $|V'_2|$. Note that $(V_i,V_{i+1})$ is an $\eps$-regular pair in $G'$ with density at least $d$ for each $i \in [k-1]$.
As  $|V'_2|\geq \eta m/2$, Proposition~\ref{prop:regsubset} and (\ref{hier3}) therefore imply that 
$(V'_i,V'_{i+1})$ is a $\sqrt{\eps}$-regular pair in $G'$ with density at least $d/2$ for each $i \in [k-1]$.

Set $\eps_0:=(d/2-\sqrt{\eps})^{\Delta}/(2+\Delta)$ where $\Delta$ is the maximum degree of $H':=P(t_1-1, t_2, \dots, t_k)$. By Lemma~\ref{simo} there are at least $(\eps _0 \eta m/2)^{h-1}$ copies of $H'$ in $G[V'_1\cup \dots \cup V'_k]$. Moreover, as $V'_2 \subseteq N^+_G(x) \cap N^+_G(y)$, each such copy of $H'$ together with $x$ forms a copy $H_x$ of $H$ in $G$, and each such copy of $H'$ together with $y$ forms a copy $H_y$ of $H$ in $G$.

Note that $$(\eps _0 \eta m/2)^{h-1}\stackrel{(\ref{hier3})}{\geq} (\eps _0 \eta n/(4L))^{h-1}.$$ 
Set $\gamma := (\eps _0 \eta /(4L))^{h-1}$ and $\xi := (\gamma /2)^h/4$.
We have shown that given any $x,y \in V(G)$, there are at least $\gamma n^{h-1}$
sets $X \subseteq V(G)$ of size $h-1$ such that both $G[X \cup \{x\}]$ and $G[X \cup \{y\}]$ 
contain spanning copies of $H$. 
Lemma~\ref{lolemma} now implies 
$V(G)$ contains a set $M$ so that 
\begin{itemize}
\item $|M|\leq \xi n$;
\item $M$ is an $H$-absorbing set for any $W \subseteq V(G) \setminus M$ such that $|W| \in h\mathbb N$ and  $|W|\leq \xi ^2 hn/2 $.
\end{itemize}
\end{proof}

With Theorem~\ref{approxmainthm} and Lemma~\ref{absorbmainthm} at hand, it is now straightforward to prove
Theorem~\ref{mainthm}.

\begin{proof}[Proof of Theorem~\ref{mainthm}]
Let $H$ be as in the statement of the theorem and set $h:=|H|$. 
Recall that it suffices to prove the case when 
 $H=P(t_1, \dots, t_k)$ for some 
 $k \geq 2$ and $t_1,\dots, t_k \in \mathbb N$.
Define  constants $\eta _1 , \xi >0$ and  $n_0\in \mathbb N$ so that
\begin{align}\label{hier4}
    \frac{1}{n_0} \ll \eta _1 \ll \xi \ll \eta, \frac{1}{h}.
\end{align}
Let $G$ be a digraph on $n \geq n_0$ vertices as in the statement of the theorem.
By Lemma~\ref{absorbmainthm},
$V(G)$ contains a set $M$ so that 
\begin{itemize}
\item $|M|\leq \xi n$;
\item $M$ is an $H$-absorbing set for any $W \subseteq V(G) \setminus M$ such that $|W| \in h\mathbb N$ and  $|W|\leq \xi ^2 hn/2 $.
\end{itemize}

Set $G':=G\setminus M$. By (\ref{hier4}), $\delta (G') \geq (1+\eta _1)|G'|$. Thus, $G'$ contains an $H$-tiling $\mathcal H_1$ covering  all but at most $\eta _1 |G'| \leq \xi ^2 hn/2$ of its vertices.
Moreover, the choice of $M$ ensures that there is an $H$-tiling $\mathcal H_2$ covering precisely those vertices in $G$ that are not in $\mathcal H_1$. Thus, $\mathcal H_1 \cup \mathcal H_2$ is our desired $H$-factor in $G$.
\end{proof}

\section{Proof of the EL-Zahar-type result}\label{sec:elzahar}
\subsection{Auxiliary results for the proof of Theorem~\ref{thm:elzthm}}
In this subsection we provide a few results that are needed for the proof of Theorem~\ref{thm:elzthm}. The first result provides a minimum semi-degree condition for forcing every orientation of a Hamilton cycle. Note that 
an asymptotic version of this result (see \cite{haggkvist1995oriented}) is also sufficient for our purposes.

\begin{theorem}\cite{arbor_hamcyc}\label{thm:oriented_ham_cycle}
  There exists $n_{\ref{thm:oriented_ham_cycle}}\in \mathbb N$ such that when 
  $n \ge n_{\ref{thm:oriented_ham_cycle}}$, 
  the following holds for every $n$-vertex digraph $G$. 
  If $\delta^0(G) > n/2$, then $G$ contains every orientation of a Hamilton cycle.
\end{theorem}

The following lemma is the main tool in our proof of Theorem~\ref{thm:elzthm}.
The iterative approach is very similar to an argument in \cite{EM}.
\begin{lemma}\label{lem:elzlem}
  For every $L \in \N$ and $\eta, \gamma > 0$, 
  there exists $n_{\ref{lem:elzlem}} := n_{\ref{lem:elzlem}}(L, \eta, \gamma)\in \mathbb N$ such that 
  for every integer $n \ge n_{\ref{lem:elzlem}}$ the following holds.
  Let $\ell_1, \dotsc, \ell_t \ge 3$ be integers of size at most $L$
  such that 
  \begin{equation}\label{eq:size_of_n}
    \ell_1 + \dotsm + \ell_t \le (1 - \gamma)n
  \end{equation}
  and let $D_1, \dotsc, D_t$ be orientations of cycles of lengths
  $\ell_1, \dotsc, \ell_t$, respectively.
  If $G$ is an $n$-vertex digraph with
  \begin{equation*}
    \delta^0(G) \ge \frac{n + t}{2} + \eta n,
  \end{equation*}
  then $G$ contains vertex-disjoint copies of $D_1, \dotsc, D_t$.
  Furthermore, if we let $G'$ be the digraph formed by deleting the vertices of these cycles from $G$, then we have 
  \[
    \delta^0(G') \ge \left( \frac{1}{2} + \frac{\eta}{3} \right)|G'|.
  \]
\end{lemma}
\begin{proof}
  For every $\ell \in [L]$, let $p_\ell$ be the number of 
  cycles in $D_1, \dotsc, D_t$ of length $\ell$.
    Define $\sigma := \eta/3$ and 
    let $\beta, \gamma > 0$ and $R \in \N$ be such that 
    \begin{equation}
      \label{eq:constants_lem}
      0 < \frac{1}{n} \ll \beta \ll \frac{1}{R} \ll \gamma, \sigma, \frac{1}{L}.
    \end{equation}

    Let $G$ be an $n$-vertex digraph as in the statement of the lemma.
    Then (\ref{eq:constants_lem}) implies that $\sum_{\ell \in [L]} R \ell < \sigma n$, so we can greedily construct a collection of vertex-disjoint cycles in $G$
    consisting of $R$ cycles of every length $\ell \in [L]$ for which $p_\ell > 0$.
    So, for notational convenience,  we can assume that $p_\ell$ is divisible by $R$ for every
    $\ell \in [L]$ and
    \begin{equation}\label{eq:min_degree_condition}
    \delta^0(G) \ge \frac{n + t}{2} + 2 \sigma n.
    \end{equation}
    For every $\ell \in [L]$, let $q_\ell := p_\ell/R$ and let $q := (q_1, \dotsc, q_L)$.
    Define 
    \[
       |q| := \sum_{i=1}^{L} q_i,
    \]
    and note that $|q| = t/R$.
    We call a collection of vertex-disjoint cycles in $G$ a \textit{$q$-cycle tiling} if it consists of exactly 
    $q_\ell$ cycles of length $\ell$ for every $\ell \in [L]$.
    We will construct the desired cycles in $R$ rounds.
    In each round $r \in [R]$, we will find a $q$-cycle tiling that is 
    vertex-disjoint from the $q$-cycle tilings constructed in the $r-1$ previous rounds.

    We can assume that
      \begin{equation}
        \label{eq:q_lower_bound}
        |q| \ge \beta n,
      \end{equation}
      as otherwise 
      $\sum_{i=1}^{t} \ell_i \le L t = L R |q| \stackrel{\eqref{eq:constants_lem}}{<} \sigma n$,
    and we can greedily find the desired vertex-disjoint cycles.

    Let 
    $m := \floor{\frac{n}{R+1}}$ and
   note that 
    \begin{equation}\label{eq:m_lower_bound}
    m 
    \ge
    \frac{n-R}{R+1} 
    = \frac{n}{R} - \frac{n}{R(R+1)} - \frac{R}{R+1},
    \end{equation}
    so 
      \begin{equation}
        \label{eq:size_of_q_cycle_tiling}
        m 
        \stackrel{\eqref{eq:size_of_n},\eqref{eq:m_lower_bound}}{\ge} 
        \frac{\sum_{i = 1}^{t}\ell_i}{R} + \frac{\gamma n}{R} - \frac{n}{R(R+1)} - \frac{R}{R+1}
        \stackrel{\eqref{eq:constants_lem}}{\ge} 
        \frac{\sum_{i=1}^{t} \ell_i}{R} = 
        \sum_{\ell \in [L]} \ell \cdot q_{\ell}. 
      \end{equation}
    
    Uniformly at random select a partition of $V(G)$ into $R+1$ parts $W_1, \dotsc, W_{R+1}$ 
    so that 
    \[
    |W_1| = |W_2| = \dotsm = |W_R| = m \quad \text{and} \quad |W_{R+1}| = n - Rm \ge m.
    \]
    For every $r \in [R+1]$, we have 
    \[
      \frac{|W_r| \cdot t}{n} \ge 
      \frac{mt}{n} \stackrel{\eqref{eq:m_lower_bound}}{\ge} 
      \frac{t}{R} - \frac{t}{R(R+1)} - \frac{tR}{n(R+1)}
      \ge |q| - \frac{n}{R(R+1)} - \frac{R}{R+1}
      \stackrel{\eqref{eq:constants_lem}}{\ge} |q| - 0.1 \sigma |W_r|,
    \]
    so by the Chernoff and union bounds the following holds with high probability
      for every $v \in V(G)$:
      \begin{equation}
        \label{eq:mindegree_to_Wit}
        d^\pm_G(v, W_{r}) \stackrel{\eqref{eq:min_degree_condition}}{\ge} 
        \frac{|W_r|}{n} \left(\frac{n + t}{2} + 2 \sigma n\right) - 0.1 \sigma |W_r|
        \ge \frac{|W_r| + |q|}{2} + \sigma |W_r|.
      \end{equation}
    The proof of the lemma follows easily once the following claim is established.
  \begin{claim}
    \label{clm:main}
    For $r \in [R+1]$, let $U \subseteq V(G)\setminus  W_r$ where $|U| \ge m$.
    Suppose that for every $v \in V(G)$ we have $d_G ^\pm(v, U) \ge (1/2 + \sigma)|U|$.
    Then $G[U \cup W_r]$ contains every orientation of a $q$-cycle tiling.
    Furthermore, if $U'$ is the set of vertices in $U \cup W_r$ that are not covered
    by the $q$-cycle tiling, then for every $v \in V(G)$ we have
    $d_G ^\pm(v, U') \ge \left(1/2 + \sigma\right)|U'|$.
  \end{claim}
  \begin{proofclaim}
        Recall that $|U| \ge m \stackrel{\eqref{eq:size_of_q_cycle_tiling}}{\ge} \sum_{\ell \in [L]} \ell q_\ell$.
        Uniformly at random select a partition $\{U_0, U_1\}$ of $U$
        such that $|U_1| = \sum_{\ell \in [L]} (\ell-1) q_\ell$. 
       Note that 
        \[
        |U_1|, |U_0| \ge |q| \stackrel{\eqref{eq:q_lower_bound}}{\ge} \beta n, 
        \]
        so the Chernoff and union bounds imply that,
        with high probability,
        for every $v \in V(G)$ and $i \in \{0,1\}$,
        \begin{equation} 
          \label{eq:minsemideg_Ui}
          d_G ^{\pm}(v, U_i) > \left(\frac{1}{2} + \sigma \right)|U_i| - \beta \sigma n 
          \stackrel{\eqref{eq:q_lower_bound}}{\ge}
          \left(\frac{1}{2} + \sigma\right)|U_i| - \sigma |q|.
        \end{equation}
         Recall that $|U_1| = \sum_{\ell \in [L]} (\ell - 1)q_\ell \ge 2|q|$.
        So \eqref{eq:minsemideg_Ui} implies
         that $\delta^0(G[U_1]) \ge (1/2 + \sigma/2)|U_1|$. 
        Theorem~\ref{thm:oriented_ham_cycle} then implies that $G[U_1]$ 
        contains a vertex-disjoint collection of paths $\mathcal{P}$ 
        such that for every $\ell \in [L]$ the collection
        $\mathcal{P}$ contains exactly $q_\ell$ paths on $\ell - 1$ vertices where
        the paths are of any desired orientation.
        Note that the paths in $\mathcal{P}$ completely cover the vertices in $U_1$.

      For every $P \in \mathcal{P}$, and 
      every $\diamond, \circ \in \{-, +\}$, 
      when we let $u,v$ be the endpoints of $P$ we have
      \[
        |N_G ^\diamond(u, W_{r}) \cap N_G ^\circ(v, W_{r})| 
        \stackrel{\eqref{eq:mindegree_to_Wit}}{\ge}
        2 \left( \frac{|W_r| + |q|}{2} + \sigma |W_r| \right) - |W_r| 
        =|q| + 2\sigma |W_r| .
      \]
      Therefore, because $|\mathcal{P}| = \sum_{\ell \in [L]} q_\ell = |q|$,
       we can greedily extend each of the $|q|$ paths in $\mathcal{P}$
       into vertex-disjoint cycles of the desired orientations using one additional vertex in $W_{r}$.
       Let $\mathcal{C}$ be the resulting $q$-cycle tiling and
         let $U' := (U \cup W_r) \setminus V(\mathcal{C})$.
      Recall that, by construction, we have $|W_r \cap U'| = |W_r| - |q|$ and $U \cap U' = U_0$.
      Therefore, for every $v \in V(H)$ and $\diamond \in \{-, +\}$ 
      \begin{equation*} 
        \begin{split} 
          d_G^\diamond(v, U') &= d_G^\diamond(v, U_0) + d_G^\diamond(v, W_r \cap U') \\ 
          &\stackrel{\eqref{eq:minsemideg_Ui}}{\ge} \left(\frac{1}{2} + \sigma\right)|U_0| - \sigma |q| + d_G^\diamond(v, W_r) - |q| \\
          &\stackrel{\eqref{eq:mindegree_to_Wit}}{\ge} \left(\frac{1}{2} + \sigma\right)|U_0| - \sigma |q| + \left(\frac{|W_r| + |q|}{2} + \sigma |W_r|\right) - |q| \\
          &= \frac{|U_0|}{2} + \frac{|W_r| - |q|}{2} + \sigma|U_0| - \sigma |q| + \sigma |W_r| \\
          &= \left(\frac{1}{2} + \sigma\right)|U'|.
        \end{split}
      \end{equation*}\qedclaim
  \end{proofclaim}
  \medskip

  Recall that \eqref{eq:mindegree_to_Wit} implies that for every $v \in V(G)$ we have $d_G ^\pm(v, W_1) \ge (1/2 + \sigma)|W_1|$. 
  Therefore, we can apply
  Claim~\ref{clm:main} with $W_1$ playing the role of $U$ and $r = 2$ to construct a $q$-tiling $\mathcal{C}_1$
  in the digraph induced by $W_1 \cup W_2$.
  Let $U' \subseteq W_1 \cup W_2$ be the set of vertices not covered by this $q$-tiling.
  Note that 
  \[
  |U'| \ge 2m - |V(\mathcal{C}_1)| = 2m - \sum_{\ell \in [L]}\ell q_\ell \stackrel{\eqref{eq:size_of_q_cycle_tiling}}{\ge} m,
  \]
  and recall that Claim~\ref{clm:main} implies that
  for every $v \in V(G)$, we have 
  $d_G ^\pm(v, U') \ge (1/2 + \sigma)|U'|$.
  Therefore, we can apply Claim~\ref{clm:main} again 
    with $U'$ now playing the role of $U$ and $r=3$ to construct
    a second $q$-tiling $\mathcal{C}_2$ in $U' \cup W_3 \subseteq W_1 \cup W_2 \cup W_3$.
   We can proceed in this manner to construct $R$ vertex-disjoint $q$-cycle tilings.
  These $R$ vertex-disjoint $q$-cycle tilings form the desired vertex-disjoint copies of $D_1,\dots, D_t$ in $G$. \qedhere
\end{proof}

We will use the following straightforward application of the Chernoff bound.
\begin{lemma}\label{lem:partition}
For every $\sigma > 0$ there exists 
$n_{\ref{lem:partition}} := n_{\ref{lem:partition}}(\sigma)\in \mathbb N$ such that for every $n \ge n_{\ref{lem:partition}}$ 
the following holds for every $n$-vertex digraph $G$ and every integer $m$
such that $\sigma n \le m \le (1 - \sigma)n$. 
There exists a partition $\{U_1, U_2\}$ of $V(G)$ such that 
$|U_1| = m$ and $|U_2| = n - m$
where, for $j \in [2]$, we have
\begin{equation*}
\delta^0(G[U_j]) 
\ge \frac{\delta^0(G)}{n} |U_j| - |U_j|^{2/3}.
\end{equation*}
\end{lemma}
\begin{proof}
We select $\{U_1,U_2\}$ uniformly at random from all partitions  of $V(G)$ with 
$|U_1| = m$ and $|U_2| = n - m$.
Then, for every $v \in V(G)$, $j \in \{1,2\}$ and $\diamond \in \{-, +\}$, 
the random variable $|N_G ^\diamond(v) \cap U_j|$ has hypergeometric distribution with expected value
$\mu := (d_G^\diamond(v)/n )\cdot |U_j|$.
Since $\mu \le |U_j|$ and $|U_j| \ge \sigma n$, by Lemma~\ref{chernoff} we have
\[
\mathbb{P}
\left[ 
\left| \left|N_G ^\diamond(v) \cap U_j\right| < \mu - |U_j|^{2/3} \right]
\right] 
\stackrel{\eqref{eq:chernoff}}{\le} e^{-|U_j|^{4/3}/2\mu} \le e^{-|U_j|^{1/3}/2} \le e^{- (\sigma n)^{1/3}/2}.
\]
Since  
$4n \cdot e^{- (\sigma n)^{1/3}/2} < 1$, the union bound implies that there exists an outcome 
where $|N_G^\diamond(v) \cap U_j| \ge \mu - |U_j|^{2/3}$ for each of the $n \cdot 2 \cdot 2$ such random variables.
\end{proof}

The following argument is essentially identical to the proof of Proposition 5.1 in \cite{alon19962} except that we use Theorem~\ref{thm:oriented_ham_cycle}
instead of Dirac's theorem.  We reproduce it here for completeness.
\begin{lemma}\label{lem:large_cycles}
  For every $\eta>0$, there exists $L := L_{\ref{lem:large_cycles}}(\eta)\in \mathbb N$ such that
  the following holds.
  Let $\ell_1, \dotsc, \ell_t$ be a sequence of integers that are each at least $L$, 
  let $n := \ell_1 + \dotsm + \ell_t$,
  and let $D_1,\dotsc,D_t$ be orientations of cycles of lengths $\ell_1, \dotsc, \ell_t$, respectively.
  If $G$ is an $n$-vertex digraph with $\delta^0(G) \ge (1/2 + \eta)n$, then 
  $G$ contains vertex-disjoint copies of $D_1, \dotsc, D_t$.
\end{lemma}
\begin{proof} 
  Let $\sigma := \min\{\eta/2, 1/3\}$.
  We assume that
  \begin{equation}
    \label{eq:consts}
    \frac{1}{L} \ll \sigma.
  \end{equation}
  In particular, 
  we assume that $L \ge n_{\ref{thm:oriented_ham_cycle}}$ and $L \ge n_{\ref{lem:partition}}(\sigma)$.
  Therefore, Lemma~\ref{lem:partition} implies that
  for every $U \subseteq V(G)$ with $|U| \ge L$ and for every integer $m$ such that 
  $\sigma|U| \le m \le (1 - \sigma) |U|$,
  there exists a partition $\{U_1, U_2\}$ of $U$ such that $|U_1| = m$, $|U_2| = |U| - m$,
  and for $j \in [2]$
  \begin{equation}\label{eq:partition_semideg}
    \delta^0(G[U_j]) 
    \ge \frac{\delta^0(G[U])}{|U|} |U_j| - |U_j|^{2/3}
    = \left(\frac{\delta^0(G[U])}{|U|} - \frac{1}{|U_j|^{1/3}} \right)|U_j|.
  \end{equation}

  The lemma follows from the following claim (with $I := [t]$, $U := V(G)$, and $k := 0$).
  \begin{claim}
    Let $I \subseteq [t]$ be non-empty and let $U \subseteq V(G)$ 
    such that $\sum_{i \in I} \ell_i = |U|$.
    If there exists a non-negative integer $k$ such that 
    \begin{equation}
      \label{eq:error}
      \delta^0(G[U])  \ge 
      \left(\frac{1}{2} + 2\sigma - \sum_{i=0}^{k-1}\left(\frac{(1-\sigma)^{i}}{|U|}\right)^{1/3}\right)|U|,
    \end{equation}
    then there exists a collection $\{ C_i \}_{i \in I}$ of vertex-disjoint cycles in $G[U]$ where $C_i$ is a copy of $D_i$ for each $i \in I$.
  \end{claim}
  \begin{proofclaim}
    We will prove the claim by induction on $|I|$.
    Since $I$ is nonempty we have $|U| \ge L$, so
      \[
        \sum_{i=0}^{k-1}\left(\frac{(1 - \sigma)^{i}}{|U|}\right)^{1/3}
        \le  \frac{1}{L^{1/3}} \sum_{i=0}^{\infty} \left(\left(1 - \sigma\right)^{1/3}\right)^i
        =  \frac{1}{L^{1/3}\left(1-\left(1 - \sigma\right)^{1/3}\right)} 
        \stackrel{\eqref{eq:consts}}{\le} \sigma,
      \]
      which implies
      \begin{equation}\label{eq:Usemideg}
        \delta^0(G[U]) \stackrel{\eqref{eq:error}}{\ge} (1/2 + \sigma)|U|.
      \end{equation}
    For convenience, we can assume $I = \{1, \dotsc, |I|\}$
      and $\ell_1 \ge \ell_2 \ge \dotsm \ge \ell_{|I|}$.
    If $\ell_1 > (1 - \sigma)|U|$, 
      then $\ell_2 + \dotsc + \ell_{|I|} = |U| - \ell_1 < \sigma |U|$, and we can
      greedily construct a (empty when $|I| = 1$) collection of 
      vertex-disjoint copies $C_2, \dotsc, C_{|I|}$ of $D_2,\dots, D_{|I|}$, 
      respectively.
    We can then use 
      Theorem~\ref{thm:oriented_ham_cycle} to find a spanning copy $C_1$ of $D_1$
       in $G[U \setminus \bigcup_{i = 2}^{|I|} V(C_i)]$.
    Note that this proves the base case ($|I| = 1$) of the induction.
    So assume $\ell_{|I|} \le \ell_{|I|-1} \le \dotsm \le \ell_1 \le (1 - \sigma)|U|$.
    Then, because $\sigma \le 1/3$, there exists a partition $\{I_1, I_2\}$ of $I$
      such that for $j \in [2]$ we have
      \[
        \sigma |U| \le \sum_{i \in I_j} \ell_i \le (1 - \sigma) |U|.
      \]
    With \eqref{eq:partition_semideg} this implies that
      there exists a partition $\{U_1, U_2\}$ of $U$ such that for $j \in [2]$
      \begin{equation}\label{eq:size_of_Uj}
        \sigma |U| \le |U_j| = \sum_{i \in I_j} \ell_i  \le (1 - \sigma)|U|
      \end{equation}
      and
      \[
        \begin{split}
          \delta^0(G[U_j]) 
          & \stackrel{\eqref{eq:partition_semideg}, \eqref{eq:error}}{\ge}
          \left(\frac{1}{2} + 2\sigma - \sum_{i=0}^{k-1}\left(\frac{(1 - \sigma)^{i}}{|U|}\right)^{1/3} - \frac{1}{|U_j|^{1/3}}\right)|U_j| \\
          & \stackrel{\eqref{eq:size_of_Uj}}{\ge} \left(\frac{1}{2} + 2\sigma - \sum_{i=0}^{k-1}\left(\frac{(1 - \sigma)^{i}}{\frac{|U_j|}{1 - \sigma}}\right)^{1/3} - \frac{1}{|U_j|^{1/3}}\right)|U_j| \\
          & = 
          \left(\frac{1}{2} + 2\sigma - \sum_{i=0}^{k}\left(\frac{(1- \sigma)^{i}}{|U_j|}\right)^{1/3}\right)|U_j|. 
        \end{split}
      \]
    The claim then follows by the induction hypothesis applied to both 
      $G[U_1]$ and $G[U_2]$.\qedclaim
\end{proofclaim}
\end{proof}

\subsection{An absorbing lemma for Theorem~\ref{thm:elzthm}}
The following absorbing lemma follows quickly from Lemma~\ref{lolemma}.
\begin{lemma}\label{lem:clabs} 
  Let $C$ be an orientation of a cycle on $\ell\geq 3$ vertices.
  Given any $\eta > 0$, there exists $\xi_{\ref{lem:clabs}} := \xi_{\ref{lem:clabs}}(\eta, \ell)>0$ such that,
  for every $0 < \xi \le \xi_{\ref{lem:clabs}}$, there exists
  $n_{\ref{lem:clabs}} = n_{\ref{lem:clabs}}(\xi, \eta, \ell) \in \mathbb N$ such that the following holds for all $n \ge n_{\ref{lem:clabs}}$
  and every $n$-vertex digraph $G$.
  If either $\ell = 3$ and 
  \[
    \delta^0(G) \ge \left(\frac{2}{3} + \eta\right)n,
  \]
  or $\ell \neq 3$ and 
  \[
    \delta^0(G) \ge \left(\frac{1}{2} + \eta\right)n,
  \]
  then $V(G)$ contains a set $M$ so that 
  \begin{itemize}
    \item $|M| \le \xi n$;
    \item $M$ is a $C$-absorbing set for any $W \subseteq V(G) \setminus M$ such that
      $|W| \in \ell \mathbb N$ and $|W| \le \ell \cdot \xi^2 n/2$.
  \end{itemize}
\end{lemma}
\begin{proof}
Define $\gamma,\xi>0$ so that $\gamma \ll \eta, 1/\ell$
and $\xi := \left(\gamma/2\right)^{\ell}/4$. Let $n\in \mathbb N$ be sufficiently large.

Let $x,y \in V(G)$.
Let $\mathcal{X}$ be the collection of $(\ell-1)$-sets $X\subseteq V(G)$ such that both $G[X \cup \{x\}]$
and $G[X \cup \{y\}]$ contain a copy of $C$.

We will show that
\[
  \ell! \cdot |\mathcal{X}| \ge (\eta n)^{\ell - 1} \ge \ell! \cdot \gamma n^{\ell - 1}.
\]
By Lemma~\ref{lolemma} (with $h := \ell$ and $t:=1$) this will prove the lemma because 
\[
\left(\gamma/2\right)^{\ell} n /4 = \xi n \qquad\text{and}\qquad
\left(\gamma/2\right)^{2\ell} \ell n / 32 = \ell \cdot \xi^2 n/2.
\]

First suppose $\ell = 3$; so $\delta^0(G) \ge \left( \frac{2}{3} + \eta\right)n$.
For every $\diamond \in \{-, +\}$, we have 
$|N_G ^\diamond(x) \cap N_G^\diamond(y) | \ge (1/3 + 2 \eta) n > \eta n$.
Furthermore, for every $z \in N_G^\diamond(x) \cap N_G^\diamond(y)$ and $\circ, \bullet \in \{+, -\}$,
we have $|N_G^\circ(x) \cap N_G^\circ(y) \cap N_G ^\bullet(z)| > \eta n$.
This implies that $2 |\mathcal{X}| \ge (\eta n)^2$, which completes the proof in this case.

Now assume $\ell \ge 4$. Using the fact that $\delta^0(G) \ge \left( \frac{1}{2} + \eta\right)n$,
we will show that we can construct at least $(\eta n)^{\ell-1}$ paths $v_1, \dotsc, v_{\ell-1}$ in $G$
so that $\{v_1, \dotsc, v_{\ell-1}\} \in \mathcal{X}$.
This will then imply that $\ell! \cdot |\mathcal{X}| \ge (\eta n)^{\ell - 1}$, which will complete the proof.

For every $\diamond \in \{-, +\}$, we have 
$|N_G^\diamond(x) \cap N_G^\diamond(y)| \ge 2 \eta n$,
so we start the construction by picking $v_1$ and then $v_{\ell - 1}$ distinct from $v_1$ in at least $(\eta n)^2$ ways.
We can iteratively select 
$v_2, \dotsc, v_{\ell - 3}$ in at least $(\eta n)^{\ell - 4}$ ways. (When $\ell = 4$, this sequence
of vertices is empty.)
Finally, because $|N_G ^\diamond(v_{\ell-3}) \cap N_G ^\circ(v_{\ell - 1})| \ge 2 \eta n$,
for every $\diamond, \circ \in \{-, +\}$, we can complete the construction by selecting $v_{\ell-2}$
in one of at least $\eta n$ ways.
As this gives us $(\eta n)^2 \cdot (\eta n)^{\ell - 4} \cdot \eta n = (\eta n)^{\ell - 1}$
ways of constructing such a path, the proof is complete.
\end{proof}

\subsection{Proof of Theorem~\ref{thm:elzthm}}
    Let $L \in \N$ and $\gamma, \xi > 0$ be such that 
    \begin{equation}\label{eq:conststhm}
      0 < \frac{1}{n} \ll \gamma \ll \xi \ll \frac{1}{L} \ll \eta.
    \end{equation}
    Let $G$ be an $n$-vertex digraph as in the statement of the theorem.
    
    We will call a collection of vertex-disjoint cycles a \textit{cycle tiling}
    and we will say two cycle tilings, $\mathcal{C}$ and $\mathcal{C'}$,
  are \textit{isomorphic} if there exist a bijection $f: \mathcal{C} \to \mathcal{C'}$ 
  such that $C$ is isomorphic to $f(C)$ for every $C \in \mathcal{C}$.
    Let $\mathcal{C}$ be a cycle tiling consisting of vertex-disjoint copies of $D_1, \dotsc, D_t$.
    Let $\Csm$ be the collection of (small) cycles in $\mathcal{C}$
      with length at most $L$ and let $\Clg := \mathcal{C} \setminus \Csm$ be the remaining (large) cycles in $\mathcal{C}$.
   \medskip

    \noindent\textbf{Case 1: } $\sum_{C \in \Clg} |C| \ge \gamma n$.
    In this case, we have
    \[
    \sum_{C \in \Csm}|C| \le n - \sum_{C \in \Clg}|C| \le (1 - \gamma)n,
    \]
    and, by \eqref{eq:conststhm}, we can assume  $n \ge N_{\ref{lem:elzlem}}(L, \eta, \gamma)$.
    Therefore, we can apply Lemma~\ref{lem:elzlem} to construct a cycle tiling isomorphic to $\Csm$.
    By Lemma~\ref{lem:elzlem}, if we let $G'$ be the digraph formed by removing the vertices of such a cycle tiling from $G$,
    we have $\delta^0(G') \ge (1/2 + \eta/3)|G'|$. 
    By \eqref{eq:conststhm} we can assume that $L \ge  L_{\ref{lem:large_cycles}}(\eta/3)$,
    so Lemma~\ref{lem:large_cycles} implies that there exist a 
    cycle tiling in $G'$ that is isomorphic to $\Clg$.
    Together, this ensures our desired copies of 
    $D_1,\dots, D_t$ in $G$.
    \medskip

    \noindent \textbf{Case 2: } $\sum_{C \in \Clg}|C| < \gamma n$.
     Let $\Ctri$ be the orientations of a triangle that appear in $\Csm$.
     If $|\Ctri| \ge n/3 - \eta n$, then 
     let $C^*$ be the orientation of a triangle that occurs most frequently among
     the cycles in $\Ctri$.
     Otherwise, let $C^*$ be the orientation of a cycle that occurs most frequently among
    the cycles in $\Csm \setminus \Ctri$.
    Let $\ell^*$ be the length of $C^*$.

   Note that 
   we clearly have $\delta^0(G) \ge (1/2 + \eta)n$ and 
   if $C^*$ is an orientation of triangle, then we have 
     \[
     \delta^0(G) 
     \ge \frac{n + |\Ctri|}{2} + \eta n
     \ge \frac{n + n/3 - \eta n}{2} + \eta n
     = \left(\frac{2}{3} + \eta/2\right)n.
     \] 
   Since $\ell^* \le L$, 
    \eqref{eq:conststhm} implies that we can assume $\xi \le \xi_{\ref{lem:clabs}}(\eta/2, \ell^*)$ and 
      $n \ge n_{\ref{lem:clabs}}(\xi, \eta/2, \ell^*)$, so, in all cases, Lemma~\ref{lem:clabs}
      implies that there exists $M \subseteq V(G)$ such that 
    \begin{equation}\label{eq:size_of_M}
    |M| \le \xi n
    \end{equation}
    and
    \begin{equation}\label{eq:absprop}
      \text{$M$ is a $C^*$-absorbing set for every $W \subseteq V(G) \setminus M$
        with $|W| \in \ell^* \N$ and $|W| \le \ell^* \cdot \xi^2 n/2$.}
    \end{equation}

    If $C^*$ is an orientation of a triangle, then $C^*$ appears at least $(1/3 - \eta )n/2 \ge \xi n$
    times in $\Ctri$.
    Otherwise, $C^*$ occurs at least 
      \begin{equation*}\label{eq:xisize}
     \frac{\sum_{C \in \Csm \setminus \Ctri}|C|}{L^2 \cdot 2^L} =
     \frac{n - 3 |\Ctri| - \sum_{C \in \Clg}|C|}{L^2 \cdot 2^L} 
     >
     \frac{3\eta n - \gamma n}{L^2 \cdot 2^L} \stackrel{\eqref{eq:conststhm}}{\ge}  
     \xi n 
      \end{equation*}
   times in $\Csm$.
    Therefore, in all cases, we can form $\Csm'$ by removing exactly 
      \[
        \frac{|M|}{\ell^*} + \floor{\xi^2 n/2} \stackrel{\eqref{eq:conststhm}, \eqref{eq:size_of_M}}{<} \xi n
      \]
      cycles that are isomorphic to $C^*$ from $\Csm$.
      (Recall that \eqref{eq:absprop} implies that $|M|$ is a divisible by $\ell^*$.)

    By \eqref{eq:size_of_M} and as we are in Case~2, we can greedily find a collection of disjoint cycles in $G\setminus M$ that
    is isomorphic to $\Clg$.
    Form $G'$ be removing the vertices of these cycles from $G \setminus M$.
  The definitions imply that 
   \begin{equation}\label{eq:size_of_Csmpp}
     \left(\sum_{C \in \Csm'} |C|\right) + \ell^* \floor{\xi^2 n/2}
     = \left(\sum_{C \in \Csm} |C|\right) - |M| 
     = |G'|.
   \end{equation}
    In particular,
    \[
      (1 - \gamma) |G'| \stackrel{\eqref{eq:conststhm}}{\ge} |G'| - \ell^* \floor{ \xi^2 n/2} 
      \stackrel{\eqref{eq:size_of_Csmpp}}= \sum_{C \in \Csm'} |C|.
     \]
    As we are in Case~2, we also have
    \begin{equation}\label{eq:min_semideg_Gprime}
      \delta^0(G') \ge \frac{n + t}{2} + \eta n - |M| -  \sum_{C \in \Clg} |C|
      \stackrel{\eqref{eq:conststhm},\eqref{eq:size_of_M}}{\ge} 
      \frac{n + t}{2} + \frac{\eta}{4} n .
    \end{equation}
      By \eqref{eq:conststhm} we can assume $|G'| \ge n_{\ref{lem:elzlem}}(L, \eta/4, \gamma)$,
      so Lemma~\ref{lem:elzlem} implies that 
    there exists a collection of disjoint cycles in $G'$ that is isomorphic to $\Csm'$.
  Let $W$ be the set of vertices in $G'$ that are not covered by this collection.
    Then,
    \[
      |W| = |G'| - \sum_{C \in \Csm'} |C| \stackrel{\eqref{eq:size_of_Csmpp}}{=} 
      \ell^* \floor{\xi^2 n/2}.
    \]
      Therefore, \eqref{eq:absprop} implies that there is a 
       $C^*$-factor in $G[M \cup W]$. 
      That is, we have found a collection of exactly $(|M| + |W|)/\ell^* = |M|/\ell^* + \floor{\xi^2 n/2}$ 
      additional vertex-disjoint cycles isomorphic to $C^*$. 
      The union of this with the collection of previously constructed cycles is isomorphic to $\mathcal{C}$.
    \qed

\section{Proof of Theorem~\ref{orethm}}\label{sec:triangle}
The proof adapts that of Theorem~4.1 from~\cite{treglown}.
Let $G$ be  as in the statement of the theorem.
Let $G'$ denote the \textit{graph} on $V(G)$ where $xy \in E(G')$
precisely when $xy \in E(G)$ or $yx \in E(G)$. So $d_{G'}(x)+d_{G'}(y) \geq 4n/3-1$ for all  non-adjacent $x \ne y \in V(G')$ by (\ref{conmin}).  Theorem~\ref{orehsz} therefore implies that $G'$ contains a $K_3$-factor and so
$G$ contains a  $\{T_3, C_3\}$-factor. Let $\mathcal M$ denote the  $\{T_3,C_3\}$-factor in $G$ that contains the most copies of $T_3$. 

Suppose for a contradiction that $\mathcal M$ is not a  $T_3$-factor. Then there is a copy $C$ of $C_3$ in $\mathcal M$. Let $V(C)= \{x,y,z\}$ where $xy,yz,zx \in E(C)$. 

Note that $ yx, zy, xz \notin E(G)$ since otherwise there is a copy of $T_3$ in $G$ on  $\{x,y,z\}$,
a contradiction to the choice of $\mathcal M$. By (\ref{conmin}) this implies that
$d ^+_G (y)+ d^-_G(x), \, d ^+_G (z)+ d^-_G(y) , \, d ^+_G (x)+ d^-_G(z) \geq  4n/3-1$.
This implies that 
$d ^+_G (x)+ d^+_G(y) + d^+_G(z) \geq 2n-1$ or $d ^-_G (x)+ d^-_G(y) +d^-_G(z) \geq 2n-1$; without loss of generality, assume the former holds.

As $ yx, zy, xz \notin E(G)$, $G[x,y,z]$ contains precisely three edges.
In particular, there are at least $d ^+_G (x)+ d^+_G(y) + d^+_G(z)-3 \geq 2n-4 >6(|\mathcal M|-1)$ edges in $G$ with startpoint in $V(C)$ and endpoint in $V(G)\setminus V(C)$.
This implies that there is an element $T \in \mathcal M \setminus \{C\}$ that receives at least $7$ edges from $V(C)$ in $G$.

Hence, there is a vertex, say $x$, in $V(C)$ that sends out $3$ edges to $T$. Furthermore, $y$ and $z$ have a common outneighbour in $G$ that lies in $V(T)$.
Together, this implies that $G[V(C) \cup V(T)]$ contains two vertex-disjoint copies of $T_3$.
This yields a  $\{T_3, C_3\}$-factor in $G$ containing more copies of $T_3$ than $\mathcal M$, a contradiction. Thus, the assumption that $\mathcal M$ is not a  $T_3$-factor is false, as desired.
\qed

\begin{remark}
Notice that we can replace (\ref{conmin}) in Theorem~\ref{orethm} with any of the following conditions:
(i) $d ^+_G (x)+ d^+_G(y) \geq  4n/3-1$; (ii) $d ^-_G (x)+ d^-_G(y) \geq  4n/3-1$; (iii) 
$d ^-_G (x)+ d^+_G(y) \geq  4n/3-1$. Indeed, in each case the proof proceeds analogously.
\end{remark}

\section{Proof of Theorem~\ref{orethm2}}\label{sec:ore2}

\subsection{Finding an almost spanning $T_r$-tiling}

Given a digraph $G$ and $z \in V(G)$, the \textit{dominant degree $d^*_G(z)$ of $z$} is defined as
$d^*_G(z):=\max \{ d^+_G(z), d^-_G(z)\}$.  
Given a set $X \subseteq V(G)$, we define
$d^*_G(z,X):=\max \{ |N^+_G(z) \cap X|, |N^-_G(z)\cap X|\}$.

The following theorem will ensure a digraph $G$ as in Theorem~\ref{orethm2} contains a $T_r$-tiling covering most of the vertices of $G$.

\begin{theorem} \label{orethm3}
Let $r \geq 2$ be an integer and let $\gamma >0$. There exists $n_0=n_0(r,\gamma) \in \mathbb N$ such that the following holds for all $n \geq n_0$.
 Suppose that $G$ is an $n$-vertex digraph  so that for any $x \ne y\in V(G)$ at least one of the following conditions holds:
\begin{itemize}
    \item[(i)] $xy, yx \in E(G)$;
    \item[(ii)] $d^*_G(z) \geq (1-1/r+\gamma)n$ for some $z \in \{x,y \}$.
\end{itemize}
Then $G$ contains a $T_r$-tiling covering all but at most $\gamma n$ vertices of $V(G)$.
\end{theorem}
Notice that a digraph $G$ as in Theorem~\ref{orethm2} satisfies the hypothesis of Theorem~\ref{orethm3} with $\eta$ playing the role of $\gamma$. In fact, the hypothesis of Theorem~\ref{orethm3}
is significantly more relaxed than that of Theorem~\ref{orethm2}.
On the other hand, one cannot strengthen the conclusion of 
Theorem~\ref{orethm3} to ensure a $T_r$-factor. For example, consider the disjoint union $G$ of two complete digraphs $A,B$, so that $|A|$ and $|B|$ are not divisible by $r$. Then $G$ does not contain a $T_r$-factor even though one can choose the sizes of $A$ and $B$ so that (i) and (ii) from Theorem~\ref{orethm3} hold.


We will need the following  simple facts.
\begin{fact}\label{simplefact}
    Let $n, r \in \mathbb N$ such that $n \geq r \geq 2$. Let $G$ be an $n$-vertex digraph
    so that for each $z \in V(G)$
    \begin{align*}
     d^* _G(z) > (1-1/(r-1))n. 
\end{align*}
Then $G$ contains a copy of $T_r$. \qed
\end{fact}

\begin{fact}\label{simplefact2}
    Let $n, r \in \mathbb N$ such that $n \geq r^2 $ and $r \geq 2$. Let $G$ be an $n$-vertex digraph
    so that for any $x \ne y\in V(G)$ such that $xy \notin E(G)$
    we have that
    \begin{align}\label{new1}
     d^+ _G(x)+d^-_G(y) > 2(1-1/r)n. 
\end{align}
Then for every $z \in V(G)$, $z$ lies in a copy of $T_r$ in $G$. 
\end{fact}
\begin{proof}
    Given any $z \in V(G)$ notice that $d^+_G (z) > (1-2/r)n$.
    If $N^+_G(z)$ induces a complete subdigraph of $G$ then we obtain our desired copy of $T_r$. Otherwise, there is some $z_1 \in N^+_G(z)$ with $d^*_G (z_1) > (1-1/r)n$; without loss of generality assume that $d^+_G (z_1)=d^*_G (z_1)$. If $r=2$ then $\{z, z_1\}$ induces our desired copy of $T_r$ in $G$. If $r>2$ then 
    $|N^+_G(z) \cap N^+_G(z_1)| >(1-3/r)n\geq 0$.  If
    $N^+_G(z) \cap N^+_G(z_1)$ induces a complete subdigraph of $G$ then we obtain our desired copy of $T_r$. Otherwise, there is some $z_2 \in N^+_G(z) \cap N^+_G(z_1)$ with $d^*_G (z_2) > (1-1/r)n$.
    By repeating this process we obtain our desired copy of $T_r$ containing $z$.
\end{proof}
\begin{remark}
Notice that we can replace (\ref{new1}) in Fact~\ref{simplefact2} with any of the following conditions:
(i) $d ^+_G (x)+ d^+_G(y) > 2(1-1/r)n$; (ii) $d ^-_G (x)+ d^-_G(y) > 2(1-1/r)n$; (iii) 
$d ^-_G (x)+ d^+_G(y) > 2(1-1/r)n$.
\end{remark}

\begin{fact}\label{fact2} Suppose that $r ,t\in \mathbb N$ such that $r$ divides $t$. Then both $T_r (t)$ and $T_{r+1} (t)$ contain  $T_r$-factors.\qed
\end{fact}

To prove Theorem~\ref{orethm3} we will build up the $T_r$-tiling using a variant of an approach that was first used in~\cite{Komlos}.
The next lemma is a key tool used for this; 
its proof is similar in flavour to that of Lemma~6.4 from~\cite{dshsz}.
\begin{lemma} \label{orelemma}
Let $r \geq 2$ be an integer and let $\gamma >0$. There exists $n_0=n_0(r,\gamma) \in \mathbb N$ such that the following holds for all $n \geq n_0$.
 Suppose that $G$ is an $n$-vertex digraph possibly with loops and  such that for any $x \ne y\in V(G)$ at least one of the following conditions holds:
\begin{itemize}
    \item[(i)] $xy, yx \in E(G)$;
    \item[(ii)] $d^*_G(z) \geq (1-1/r+\gamma)n$ for some $z \in \{x,y \}$.
\end{itemize}
Further, suppose the largest $T_r$-tiling in $G$ covers $n' \leq (1-\gamma)n$ vertices of $V(G)$.
Then $G$ contains a $\{T_r,T_{r+1}\}$-tiling that covers at least $n'+\gamma ^2 n/3$ vertices of $V(G)$.
\end{lemma}
\begin{proof}
We first show that $n' \geq \gamma n/3$.
Let $K$ be  the largest complete digraph in $G$. By
conditions (i) and (ii), every set $X\subseteq V(G)$ of size at least $|K|+1$ contains a vertex
$z \in X$ so that $d^*_G(z) \geq (1-1/r+\gamma)n$. If $|K|\geq \gamma n/2$ then certainly
$G$ contains a $T_r$-tiling covering at least $|K|-r+1\geq \gamma n/3$ vertices.
Otherwise, there is a set $Y \subseteq V(G)$ where $|Y| \leq |K|<\gamma n/2$ so that
$d^*_{G}(z) \geq (1-1/r+\gamma)n$ for every $z \in V(G)\setminus Y$; in particular, 
$d^*_{G\setminus Y}(z) \geq (1-1/r+\gamma/2)n$ for every $z \in V(G)\setminus Y$.
In this case we can repeatedly apply Fact~\ref{simplefact} to obtain a  $T_r$-tiling in $G\setminus Y \subseteq G$ covering at least $\gamma  n/2$ vertices.
In either case, this shows that  $n' \geq \gamma n/3$.

Let $\mathcal M$ be the largest $T_r$-tiling in $G$ and suppose that $n' \leq (1-\gamma)n$.
Note that $G_1:=G \setminus V(\mathcal M)$ does not contain a copy of $T_r$ by the maximality of $\mathcal M$. Thus, (i) and (ii) imply that all but at most $r-1$ vertices $z \in V(G_1)$ satisfy
$d^*_G(z) \geq (1-1/r+\gamma)n$.

We claim that there are at least $\gamma ^2 n$ vertices $z \in V(G_1)$ that satisfy
$d^*_{G_1} (z) < (1-1/r+\gamma)|G_1|$. Suppose for a contradiction that this is not the case. Then
by deleting at most $\gamma ^2n \leq \gamma |G_1|$ vertices from $G_1$ we obtain an induced subdigraph $G_2$ of $G_1$
so that $d^*_{G_2} (z) > (1-1/r)|G_2|$ for every $z \in V(G_2)$; so Fact~\ref{simplefact} implies that $G_2 \subseteq G_1$ contains a copy of $T_r$, a contradiction.

The last two paragraphs imply that there are at least $\gamma ^2 n- (r-1)\geq \gamma ^2 n/3$ vertices
$z \in V(G_1)$ such that $d^*_G(z, V(\mathcal M)) \geq (1-1/r+\gamma)n- (1-1/r+\gamma)|G_1|
= (1-1/r+\gamma) |V(\mathcal M)|$.
This condition implies that for each such $z$, there are at least $\gamma |V(\mathcal M)|= \gamma n' \geq \gamma ^2 n/3$ copies $T$ of $T_r$ in $\mathcal M$ such that (a) $z$ sends out all possible edges to  $T$ or (b) $z$ receives all possible edges from  $T$; in particular, $z$ forms a copy of $T_{r+1}$ with each such $T$.

One can now greedily assign at least  $\gamma ^2 n/3$ such vertices $z$  to distinct copies of $T_r$ in $\mathcal M$ to obtain a 
$\{T_r,T_{r+1}\}$-tiling that covers at least $n'+\gamma ^2 n/3$ vertices, as desired.
\end{proof}

We will now combine Lemma~\ref{orelemma} with the diregularity lemma to prove Theorem~\ref{orethm3}.
The main novelty in our proof is that we need to consider the special reduced digraph $R^*$ of $G$, rather than the reduced digraph $R$.
For this, we will need to extend the notion of a $t$-blow-up of a digraph to digraphs with loops.
Indeed, 
given a  digraph $R^*$ with loops, the digraph (with loops) $R^*(t)$  is defined as follows:
 $V(R^*(t)):= \{v^{j} \colon v \in V(R^*) \text{ and } j\in [t]\}$ and $E(R^*(t)) := \{v^a w^{b} :  v w \in E(R^*) \text{ and }a,b\in [t]\}.$
 Thus, if there is a loop at $v \in V(R^*)$, then for every $a,b\in [t]$, $R^*(t)$ contains the edge $v^a v^{b}$ and in particular, there is  a loop at $v^a$ in $R^*(t)$.

\begin{proof}[Proof of Theorem~\ref{orethm3}]
    Define additional constants $\eps, d$ and $n_0, \ell_0 \in \mathbb N$ so that 
\begin{align}\label{orehier}
 0<\frac{1}{n_0}\ll \frac{1}{\ell_0} \ll \eps \ll d \ll \gamma , \frac{1}{r}. 
    \end{align}
    Note that it suffices to prove the theorem under the assumption that $\gamma <1/4$.
    Set $s:= \lceil 12/\gamma^2 \rceil $. Apply Lemma~\ref{lem:reg} with parameters $\eps, d$ and $\ell_0$ to $G$ 
to obtain clusters $V_1, \dots , V_k$, an exceptional set $V_0$ and a pure digraph $G'$.
Set $m:=|V_1|=\dots =|V_k|$.
 Let $R$ and $R^*$ respectively denote the reduced digraph and special reduced digraph of $G$ with parameters $\eps, d$ and $\ell_0$. 
 
Note that we may assume that for any
distinct $i,j \in [k]$, if $(V_i,V_j)$ is of density $1$ in $G$
then it is of density $1$ in $G'$. 

For any distinct $i,j \in [k]$, suppose that $V_i V_j \notin E(R)$ or $V_j V_i \notin E(R)$. Then by the remark in the last paragraph,  there exist some $x \in V_i$ and $y \in V_j$ such that $xy \notin E(G)$ or $yx \notin E(G)$. Thus, for some $z \in \{x,y\}$,
$d^*_G(z) \geq (1-1/r+\gamma)n$ by (ii) and so
$d^*_{G'}(z) \geq (1-1/r+3\gamma/4)n$ by (\ref{orehier}) and 
Lemma~\ref{lem:reg}(c). Without loss of generality, assume that $z=x$. Note that
$d^*_{G'}(x, V_1\cup \dots \cup V_k) \geq (1-1/r+\gamma/2)n$ by 
Lemma~\ref{lem:reg}(a). This together with Lemma~\ref{lem:reg}(e)
implies that $d^*_R (V_i) \geq (1-1/r+\gamma/2)n/m \geq (1-1/r+\gamma/2)k$.

In summary, for any distinct $i,j \in [k]$ at least one of the following conditions holds:
\begin{itemize}
    \item[(i$'$)] $V_iV_j, V_jV_i \in E(R)$;
    \item[(ii$'$)] $d^*_R(V_i) \geq (1-1/r+\gamma/2)k$ or $d^*_R(V_j) \geq (1-1/r+\gamma/2)k$.
\end{itemize}

Furthermore, note that if there is no loop at some $V_i \in V(R^*)$ in $R^*$, then there exists a $z \in V_i$ so that $d^*_G(z) \geq (1-1/r+\gamma)n$. Arguing as before, we conclude that
\begin{itemize}
    \item[(iii$'$)] $d^*_{R^*} (V_i) \geq (1-1/r+\gamma/2)k$ for every $V_i \in V(R^*)$ that does not have a loop in $R^*$.
\end{itemize}
Let $t \in \mathbb N$. Crucially, (i$'$)--(iii$'$)
ensure that for any distinct $X,Y \in V(R^*(t))$,
\begin{itemize}
    \item[(i$''$)] $XY, YX \in E(R^*(t))$, or
    \item[(ii$''$)] $d^*_{R^*(t)}(Z) \geq (1-1/r+\gamma/2)kt= (1-1/r+\gamma/2)|R^*(t)|$ for some $Z \in \{X,Y \}$.
\end{itemize}

\begin{claim}\label{blowclaim}
 $R^*_s:=R^*(r^s)$ contains a $T_r$-tiling covering at least $(1-\gamma /2)kr^s=(1-\gamma /2)|R^*_s|$ vertices. 
\end{claim}
\begin{proofclaim}
If $R^*$ contains a $T_r$-tiling covering at least $(1-\gamma/2)k$ vertices then Fact~\ref{fact2} implies that Claim~\ref{blowclaim} holds. Thus, suppose that the largest $T_r$-tiling in $R^*$ covers precisely
$d \leq (1- \gamma /2)k$ vertices. Then by (i$'$) and (ii$'$) we may apply Lemma~\ref{orelemma} to $R^*$ to conclude that $R^*$ contains a $\{T_r, T_{r+1} \}$-tiling that covers at least $d+ \gamma ^2 k/12$ vertices. Thus, by Fact~\ref{fact2}, $R^*(r)$ contains a $T_r$-tiling
covering at least $(d+\gamma ^2 k/12)r$ vertices. (So at least a $\gamma ^2/12$-proportion of the vertices in $R^*(r)$ are covered.) 

If $R^*(r)$ contains a $T_r$-tiling covering at least $(1-\gamma /2)kr$ vertices then again Fact~\ref{fact2} implies that the claim holds.
So suppose that the largest $T_r$-tiling in $R^*(r)$ covers precisely
$d' \leq (1- \gamma /2)kr$ vertices. Recall that $d' \geq (d+\gamma ^2 k/12)r$. 
Then by (i$''$) and (ii$''$) we may apply Lemma~\ref{orelemma} to $R^*(r)$ to conclude that $R^*(r)$ contains a
 $\{T_r, T_{r+1} \}$-tiling that covers at least $d' +\gamma ^2 kr/12 \geq (d+\gamma ^2 k/6)r$ vertices.
Thus, by Fact~\ref{fact2}, $R^*(r^2)$ contains a $T_r$-tiling
covering at least $(d+\gamma ^2 k/6)r^2$ vertices. (So at least a $\gamma^2 /6$-proportion of the vertices in $R^*(r^2)$ are covered.) Repeating this argument at most $s$ times we see that the claim holds.\qedclaim
\end{proofclaim}

\medskip

We now use Claim~\ref{blowclaim} to prove the theorem.
For each $ i \in [ k]$, partition $V_i$ into classes $V^0 _i, V_{i,1}, \dots , V_{i, r^s}$ where $m':=|V_{i,j}|= \lfloor m/r^s \rfloor \geq m/(2r^s)$ for all $ j \in [r^s]$.
Since $mk \geq (1-\eps)n$ by Lemma~\ref{lem:reg}(a),
\begin{align}\label{m'bound}
m'|R^*_s| = \big \lfloor {m}/{r^s} \big \rfloor kr^s \geq mk-kr^s \geq (1-2\eps)n.
\end{align}

Proposition~\ref{prop:regsubset} implies that, for any distinct $i_1,i_2 \in [k]$, if $(V_{i_1}, V_{i_2})_{G'}$ is $\eps$-regular with density at least $d$ then $(V_{i_1, j_1}, V_{i_2,j_2})_{G'}$ is $2\eps r^s$-regular with density at least $d-\eps \geq d/2$
for all $ j_1,j_2 \in [r^s]$. 
Moreover, for any $i \in [k]$ and any
distinct $ j_1,j_2 \in [r^s]$,
if there is a loop at $V_i$ in  $V(R^*)$,
then there are all possible edges between $V_{i, j_1}$ and $V_{i,j_2}$ in  $G$ (and so certainly
$(V_{i, j_1}, V_{i,j_2})_{G}$ is $2\eps r^s$-regular with density at least $ d/2$).

We can therefore label the vertex set of $R^*_s$ so that $V(R^*_s)=\{V_{i,j} :  i \in [k] , \ j \in [r^s] \}$ where 
$V_{i_1, j_1} V_{i_2,j_2} \in E(R^*_s)$ implies that either (a) $(V_{i_1, j_1}, V_{i_2,j_2})_{G'}$ is $2\eps r^s$-regular with density at least $d/2$ or (b) there are all possible edges between $V_{i_1, j_1}$ and $V_{i_2,j_2}$ in  $G$. 

By Claim~\ref{blowclaim}, $R^*_s$ has a $T_r$-tiling $\mathcal M$ that contains at least $(1-\gamma /2)|R^*_s|$ vertices. Consider any copy $T$ of $T_r$ in $\mathcal M$ and let
$V(T)=\{ V_{i_1, j_1}, V_{i_2,j_2}, \dots , V_{i_r, j_r} \}$. Set $V:= V_{i_1, j_1}\cup V_{i_2,j_2}\cup \dots \cup V_{i_r, j_r}$.
Note that $0<1/m' \ll 2 \eps r^s \ll d/2 \ll \gamma , 1/r$.
Lemma~\ref{newlemma} implies that 
$G[V]$ contains a $T _r$-tiling covering all but at most $\sqrt{2 \eps r ^s} m' r \leq \gamma ^2 m' r$ vertices. By considering each copy of $T_r$ in $\mathcal M$
we conclude that
$ G$ contains a $T _r$-tiling covering at least
$$ (1-\gamma ^2)m'r \times (1- \gamma /2)|R^*_s|/r \stackrel{(\ref{m'bound})}{\geq} (1-\gamma^2)(1-\gamma /2) (1-2\eps )n \stackrel{(\ref{orehier})}{\geq} (1-\gamma )n $$
vertices, as desired.

\end{proof}

\subsection{Proof of Theorem~\ref{orethm2}}
We now have all the tools required to prove Theorem~\ref{orethm2}.

\begin{proof}[Proof of Theorem~\ref{orethm2}]
    Define additional constants $\gamma, \gamma _1, \gamma _2, \eta _1>0$ and $n_0 \in \mathbb N$ such that
    \begin{align}\label{orehier1}
        \frac{1}{n_0} \ll \gamma _2 \ll \gamma _1 \ll \gamma \ll \eta _1 \ll \eta , \frac{1}{r}.
    \end{align}
    Note that it suffices to prove the theorem under the additional assumption that $\eta \ll 1/r$.

    Let $G$ be a digraph on $n \geq n_0$ vertices as in the statement of the theorem. For any $x \ne y\in V(G)$, (\ref{oreuseful}) implies that at least one of the following conditions holds:
\begin{itemize}
    \item[(i)] $xy, yx \in E(G)$;
    \item[(ii)] $d^*_G(z) \geq (1-1/r+\eta)n$ for some $z \in \{x,y \}$.
\end{itemize}
Partition $V(G)$ into $S, L$ where $S$ consists precisely of those $z \in V(G)$ such that $d^*_G(z) < (1-1/r+\eta)n$. 
Notice  (i) and (ii) imply that $G[S]$ is a complete digraph.
In particular, $|S|< (1-1/r+\eta)n+1$ and so certainly $|L| \geq 2\eta _1 n$. Our argument now splits into two cases.

\smallskip

{\bf Case 1: $|S|\geq \eta _1n$.}
Our first aim is to apply Lemma~\ref{lolemma} to obtain a $T_r$-absorbing set.

\begin{claim}\label{claimore1}
    Given any distinct $x,y \in V(G)$ so that either $x,y \in S$ or $x,y \in L$,
    there are 
at least $\gamma n^{r-1}$ $(r-1)$-sets $X \subseteq V(G)$
such that $G[X \cup \{x\}]$ and $G[X \cup \{y\}]$
contain spanning copies of $T_r$.
\end{claim}
\begin{proofclaim}
Given any distinct $x,y \in S$, since $G[S]$ is  complete, 
every $(r-1)$-set $X \subseteq S\setminus \{x,y\}$ is such that 
$G[X \cup \{x\}]$ and $G[X \cup \{y\}]$
contain spanning copies of $T_r$. There are at least
$(\eta _1 n/2)^{r-1}/(r-1)!  \stackrel{(\ref{orehier1})}{\geq} \gamma n^{r-1}$ such choices for $X$. 

Next consider any distinct $x,y \in L$. 
Then $d^*_G(x),d^*_G(y)  \geq (1-1/r+\eta)n$. Without loss of generality, suppose that $d^+_G(x)=d^*_G(x)$ and $d^-_G(y) =d^*_G(y)$; the other cases are analogous.
Then $|N^+_G (x) \cap N^-_G(y)| \geq (1-2/r +2\eta )n$.
If $|N^+_G (x) \cap N^-_G(y) \cap S| \geq \eta _1 n$, then
as in the last paragraph we conclude that there are 
at least $\gamma n^{r-1}$ $(r-1)$-sets $X \subseteq S$
such that $G[X \cup \{x\}]$ and $G[X \cup \{y\}]$
contain spanning copies of $T_r$.
So assume that $|N^+_G (x) \cap N^-_G(y) \cap S| < \eta _1 n$ and so
$|L'| \geq (1-2/r +\eta )n$ where $L':=N^+_G (x) \cap N^-_G(y) \cap L$. As  $d^*_G(z) \geq (1-1/r+\eta)n$ for all $z \in L'$, it is now straightforward to greedily construct at least 
 $\gamma n^{r-1}$ $(r-1)$-sets $X \subseteq L'$
such that $G[X]$ 
contains a spanning copy of $T_{r-1}$; in particular, for each such $X$, both $G[X \cup \{x\}]$ and $G[X \cup \{y\}]$
contain spanning copies of $T_r$.\qedclaim
\end{proofclaim}

\begin{claim}\label{claimore2}
    Given any $x \in S$, there are at least $\eta _1 n$ vertices
    $y \in L$ such that the following holds:
    there are 
at least $\gamma n^{r-1}$ $(r-1)$-sets $X \subseteq V(G)$
such that $G[X \cup \{x\}]$ and $G[X \cup \{y\}]$
contain spanning copies of $T_r$.
\end{claim}
\begin{proofclaim}
Notice that since $x \in S$, $xy \notin E(G)$ for at least $\eta _1 n$ vertices $y \in L$. Thus, by (\ref{oreuseful}) we have that
$d^+_G(x) +d^-_G(y) \geq 2(1-1/r+\eta)n$. 
This implies that 
$|N^+_G (x) \cap N^-_G(y)| \geq (1-2/r +2\eta )n$.
If $|N^+_G (x) \cap N^-_G(y) \cap S| \geq \eta _1 n$, then
arguing  as in the proof of Claim~\ref{claimore1} we conclude that there are 
at least $\gamma n^{r-1}$ $(r-1)$-sets $X \subseteq S$
such that $G[X \cup \{x\}]$ and $G[X \cup \{y\}]$
contain spanning copies of $T_r$.
Otherwise, $|N^+_G (x) \cap N^-_G(y) \cap S| < \eta _1 n$ and so
$|L'| \geq (1-2/r +\eta )n$ where $L':=N^+_G (x) \cap N^-_G(y) \cap L$. Again arguing as in the proof of Claim~\ref{claimore1}, we conclude that there are 
at least $\gamma n^{r-1}$ $(r-1)$-sets $X \subseteq L'$
such that $G[X \cup \{x\}]$ and $G[X \cup \{y\}]$
contain spanning copies of $T_r$. \qedclaim
\end{proofclaim}

The previous two claims  allow us to prove the following.
\begin{claim}\label{claimore3}
    Given any distinct $x,y \in V(G)$,
    there are 
at least $\gamma _1 n^{2r-1}$ $(2r-1)$-sets $X_1 \subseteq V(G)$
such that $G[X_1 \cup \{x\}]$ and $G[X_1 \cup \{y\}]$
contain $T_r$-factors.
\end{claim}
\begin{proofclaim}
Consider any  $x,y \in V(G)$. First, suppose that either $x,y \in S$ or $x,y \in L$. Claim~\ref{claimore1} implies that there are at least $\gamma n^{r-1}$ $(r-1)$-sets $X \subseteq V(G)$
such that $G[X \cup \{x\}]$ and $G[X \cup \{y\}]$
contain spanning copies of $T_r$.
Fix any such $X$. Select any $r$-set $X' \subseteq S\setminus (X \cup \{x,y\})$; there are at least 
$(\eta _1 n/2)^{r}/r!$ choices for $X'$. 
Set $X_1:=X \cup X'$.
As 
$G[X']$ is a complete digraph, $G[X_1 \cup \{x\}]$ and $G[X_1 \cup \{y\}]$
contain $T_r$-factors.
In total,  there are at least
$$
\gamma n^{r-1} \times \frac{(\eta _1 n/2)^{r}}{r!} \times \frac{1}{(2r-1)!} \stackrel{(\ref{orehier1})}{\geq} \gamma _1 n^{2r-1}
$$
choices for $X_1$, as required.

Finally, consider the case when $x \in S$ and $y \in L$.
By Claim~\ref{claimore2} there are at least $\eta _1 n-1$ vertices $z \in L\setminus \{y\}$ so that the following holds:
there are 
at least $\gamma n^{r-1}- n^{r-2}\geq \gamma n^{r-1}/2$ $(r-1)$-sets $X \subseteq V(G)\setminus \{y\}$
such that $G[X \cup \{x\}]$ and $G[X \cup \{z\}]$
contain spanning copies of $T_r$.
Fix any such choice of $z$ and $X$. 
By Claim~\ref{claimore1} there are at least 
$\gamma n^{r-1}- r n^{r-2} \geq \gamma n^{r-1} /2$ $(r-1)$-sets $X' \subseteq V(G)\setminus (X \cup \{x\})$
such that $G[X' \cup \{y\}]$ and $G[X' \cup \{z\}]$
contain spanning copies of $T_r$. 
Set $X_1:=\{z\} \cup X \cup X'$. By construction of $X_1$, we have that $G[X_1 \cup \{x\}]$ and $G[X_1 \cup \{y\}]$
contain $T_r$-factors.
In total,  there are at least
$$(\eta _1n-1) \times 
\frac{\gamma n^{r-1}}{2} \times \frac{\gamma n^{r-1}}{2} \times \frac{1}{(2r-1)!} \stackrel{(\ref{orehier1})}{\geq} \gamma _1 n^{2r-1}
$$
choices for $X_1$. \qedclaim
\end{proofclaim}

\smallskip

Claim~\ref{claimore3} ensures that we can apply Lemma~\ref{lolemma} with $T_r, r, 2, \gamma _1$ playing the roles of $H, h, t, \gamma$. Thus, 
$V(G)$ contains a set $M$ so that
\begin{itemize}
\item $|M|\leq (\gamma_1/2)^r n/4$;
\item $M$ is a $T_r$-absorbing set for any $W \subseteq V(G) \setminus M$ such that $|W| \in r \mathbb N$ and  $|W|\leq (\gamma_1 /2)^{2r} rn/32 $.
\end{itemize}

Set $G_1:=G\setminus M$. Then (i) and (ii) imply that
for any $x \ne y\in V(G_1)$ at least one of the following conditions holds:
\begin{itemize}
    \item $xy, yx \in E(G_1)$;
    \item $d^*_{G_1}(z) \geq (1-1/r+\gamma_2 )|G_1|$ for some $z \in \{x,y \}$.
\end{itemize}
Theorem~\ref{orethm3} therefore implies that $G_1$ contains a $T_r$-tiling $\mathcal H_1$ covering all but at most $\gamma _2 |G_1| \stackrel{(\ref{orehier1})}{\leq} (\gamma_1 /2)^{2r} rn/32 $ vertices in $G_1$. Moreover, the choice of $M$ ensures that there is a $T_r$-tiling $\mathcal H_2$ covering precisely the vertices in $G$ that are not in  $\mathcal H_1$. Thus, $\mathcal H:=\mathcal H_1 \cup \mathcal H_2$ is our desired $T_r$-factor in $G$.

\smallskip

{\bf Case 2: $|S|<\eta _1n$.}
In this case, by repeatedly applying Fact~\ref{simplefact2} we greedily construct a $T_r$-tiling $\mathcal S$ in $G$ of size at most $\eta _1 n$ that covers all of $S$.

Let $G_0:= G \setminus V(\mathcal S)$. Note that
as $V(G_0)\subseteq L$ we have that
$d_{G_0}^* (z) \geq (1-1/r+\eta /2)n$ for every $z \in V(G_0)$.
This condition allows us to conclude that, for every distinct $x, y \in V(G_0)$,  there are at least 
 $\gamma |G_0|^{r-1}$ $(r-1)$-sets $X \subseteq V(G_0)$
such that $G_0[X \cup \{x\}]$ and $G_0[X \cup \{y\}]$
contain spanning copies of $T_r$.
We can now apply Lemma~\ref{lolemma} and argue analogously to the last two paragraphs of Case~1 to conclude that $G_0$ contains a $T_r$-factor $\mathcal H$. Thus, $\mathcal S \cup \mathcal H$ is our desired $T_r$-factor in $G$.
\end{proof}

\begin{remark}
Notice that we can replace (\ref{oreuseful}) in Theorem~\ref{orethm2} with any of the following conditions:
(i) $d ^+_G (x)+ d^+_G(y) \geq  2(1-1/r+\eta)n$; (ii) $d ^-_G (x)+ d^-_G(y) \geq   2(1-1/r+\eta)n$; (iii) 
$d ^-_G (x)+ d^+_G(y) \geq   2(1-1/r+\eta)n$. Indeed, in each case the proof proceeds analogously.
\end{remark}

\section{Open problems}\label{seccon}

In Theorem~\ref{elzaharspecial} we asymptotically determined the minimum semi-degree threshold for forcing a $C$-factor in a digraph for every orientation of an odd cycle $C$. The corresponding problem for even cycles seems more challenging;
as a starting point, it would be interesting to resolve the following conjecture.

\begin{conjecture}\label{conj1}
    Given any $\eta>0$, there exists $n_0=n_0(\eta)\in \mathbb N$ such that the following holds for all $n \geq n_0$ divisible by $4$.
If $G$ is an $n$-vertex digraph with 
$$\delta^0(G)\geq (1/2+\eta)n,$$
then $G$ contains a $C_4$-factor.
\end{conjecture}
The example given after the statement of Theorem~\ref{mainthm} again shows that if Conjecture~\ref{conj1} is true, then the bound on $\delta^0(G)$ is asymptotically best possible.
A resolution of Conjecture~\ref{conj1}  could also provide insight into the question for even cycle factors more generally.

\begin{question}\label{ques2}
    Is the following statement true?  Let $C$ be an orientation of an even cycle. Given any $\eta>0$, there exists $n_0=n_0(\eta,C)\in \mathbb N$ such that the following holds for all $n \geq n_0$ divisible by $|C|$.
If $G$ is an $n$-vertex digraph with 
$$\delta^0(G)\geq (1/2+\eta)n,$$
then $G$ contains a $C$-factor.
\end{question}

\begin{remark}
Note that Lemma~\ref{lem:clabs} implies that to resolve Conjecture~\ref{conj1} and Question~\ref{ques2}, it suffices to resolve the `almost perfect' $C$-tiling versions of them.
\end{remark}

Finally, in Theorem~\ref{orethm2} we asymptotically determined the Ore-type threshold for forcing a $T_r$-factor. It would be interesting to prove a sharp version of this result as well as to 
 establish other Ore-type conditions for forcing $H$-factors in digraphs.

\section*{Acknowledgments}
The authors are grateful to Andrzej Czygrinow, Louis DeBiasio,  Allan Lo and Reshma Ramadurai for useful discussions.

\smallskip
{\noindent \bf Data availability statement.}
There are no additional data beyond that contained within the main manuscript.

{\noindent \bf Open access statement.}
	This research was funded in part by the  EPSRC grant   UKRI1117. For the purpose of open access, a CC BY public copyright licence is applied to any Author Accepted Manuscript arising from this submission.


\begin{thebibliography}{99}




\bibitem{abbasi} S. Abbasi, The solution of the El-Zahar problem, Ph.D. Thesis, Rutgers University, 1998. 

\bibitem{alon19962}
N. Alon and E. Fischer, 2-factors in dense graphs, \emph{Discr.
  Math.} \textbf{152} (1996),  13--23.

\bibitem{AlonShapira04}
N. Alon and A. Shapira,
Testing subgraphs in directed graphs,
\textit{J. Comput. System Sci.}
{\bf 69} (2004),
353--382.


\bibitem{alony}
N. Alon and R. Yuster, Almost $H$-factors in dense graphs, \emph{Graphs Combin.} {\bf 8} (1992), 95--102.


\bibitem{cdkm} A. Czygrinow, L. DeBiasio, H.A. Kierstead and T. Molla, 
An extension of the Hajnal--Szemer\'edi theorem to directed graphs, 
{\it Combin. Probab.  Comput.} 
{\bf 24} (2015), 754--773.

\bibitem{cdmt} A. Czygrinow, L. DeBiasio, T. Molla and A. Treglown,
Tiling directed graphs with tournaments, \emph{Forum Math. Sigma} {\bf 6} (2018), e2.

\bibitem{ckm} A. Czygrinow, H.A. Kierstead and T. Molla, On directed versions of the Corr\'adi--Hajnal Corollary, \emph{Euro. J. Combin.}~{\bf 42} (2014), 1--14.

\bibitem{arbor_hamcyc}
L.~DeBiasio, D.~K\"{u}hn, T.~Molla, D.~Osthus and A.~Taylor, Arbitrary orientations of {H}amilton cycles in digraphs, \emph{SIAM J. Discr. Math.}
  \textbf{29} (2015), 1553--1584.
  
\bibitem{debiasiomolla} L. DeBiasio and T. Molla,
Semi-degree threshold for anti-directed Hamiltonian cycles, \emph{Electron. J. Combin.} {\bf{22} }(2015), P4.34.

\bibitem{zahar}
M.H. El-Zahar, On circuits in graphs, \emph{Discr. Math.} {\bf 50} (1984), 227--230.

\bibitem{enomoto} H. Enomoto, On the existence of disjoint cycles in a graph, \textit{Combinatorica} {\bf 18} (1998), 487--492.

\bibitem{EM} B.~Ergemlidze and T.~Molla, Transversal $C_k$-factors in subgraphs of the balanced blow-up of $C_k$,
{\it Combin. Probab.  Comput.} {\bf 31} (2022), 1031--1047.

\bibitem{gh} A. Ghouila-Houri, Une condition suffisante d'existence d'un circuit hamiltonien,
{\it C.R. Acad. Sci. Paris} {\bf 251} (1960), 495--497.

\bibitem{haggkvist1995oriented}
R.~H\"aggkvist and A.~Thomason, Oriented {H}amilton cycles in digraphs,
  \emph{J. Graph Theory} \textbf{19} (1995), 471--479.


\bibitem{hs} A. Hajnal and E. Szemer\'edi, Proof of a conjecture of P. Erd\H{o}s, {\it Combinatorial Theory and its Applications vol. II} {\bf 4} (1970), 601--623.


\bibitem{kier} H.A. Kierstead and A.V. Kostochka, An Ore-type Theorem on Equitable
Coloring, \emph{J. Combin. Theory Ser. B}~{\bf{98}}  (2008), 226--234.

\bibitem{Komlos} J. Koml\'os, Tiling Tur\'an theorems, {\it Combinatorica} {\bf 20} (2000), 203--218.



\bibitem{blowup} J.~Koml\'os, G.N.~S\'ark\"ozy and E.~Szemer\'edi,
Blow-up lemma,
\emph{Combinatorica} {\bf 17} (1997), 109--123.

\bibitem{simo} J. Koml\'{o}s and M. Simonovits, Szemer\'{e}di's Regularity Lemma and its applications in graph theory, \emph{Combinatorics: Paul Erd\H{o}s is eighty vol. II}
(1996), 295--352.

\bibitem{KuhnO} D. K\"uhn and D. Osthus, The minimum degree threshold for perfect graph packings, {\it Combinatorica} {\bf 29} (2009), 65--107.

\bibitem{lonote} A. Lo, From finding a spanning subgraph $H$ to an $H$-factor, arXiv:2509.18832.

\bibitem{lomark}  A. Lo and K. Markstr\"om, $F$-factors in hypergraphs via absorption, \emph{Graphs Combin.} {\bf 31} (2015), 679--712.

\bibitem{mont} A. Kathapurkar and R. Montgomery, Spanning trees in dense directed graphs, \emph{J. Combin. Theory Ser. B} {\bf 156} (2022), 223--249.

\bibitem{theothesis} T. Molla, On tiling directed graphs with cycles and tournaments, PhD thesis, Arizona State University, 2013.

\bibitem{tassio} R. Mycroft and T. Naia, Trees and tree-like structures in dense digraphs, 	arXiv:2012.09201.

\bibitem{stein} M. Stein, 
Oriented trees and paths in digraphs, \emph{Surveys in Combinatorics},
London Mathematical Society Lecture Note Series, Cambridge University Press, 2024.

\bibitem{stein2} M. Stein and C. Z\'arate-Guer\'en, Antidirected subgraphs of oriented graphs,
{\it Combin. Probab.  Comput.} {\bf 33} (2024), 446--466.

\bibitem{treglown} A. Treglown,
On directed versions of the Hajnal--Szemer\'edi theorem,
\emph{Combin. Probab. Comput.} {\bf 24} (2015), 873--928. 

\bibitem{dshsz}  A. Treglown,
A degree sequence Hajnal--Szemer\'edi theorem,
\emph{J. Combin. Theory Ser. B}
{\bf 118} (2016), 13--43.

\bibitem{wang} H. Wang, On the maximum number of disjoint cycles in a graph, \emph{Discr. Math.} {\bf 205}
(1999), 183--190.

 \bibitem{wang2} H. Wang, Independent directed triangles in a directed graph, \emph{Graphs  Combin.}~{\bf 16} (2000),
    453--462.

\bibitem{wang3} H. Wang, Disjoint directed cycles in directed graphs, \emph{Discr. Math.} {\bf 343}
(2020), 111927.

\bibitem{wood} D.R. Woodall, Sufficient conditions for circuits in graphs, \emph{Proc. London Math. Soc.} {\bf 24} (1972),
739--755.

\bibitem{zw} D. Zhang and H. Wang, Disjoint directed quadrilaterals in a directed graph,
\emph{J. Graph Theory} {\bf 50} (2005), 91--104.
\end{thebibliography}
\end{document}